\documentclass[article,12pt]{amsart}
\usepackage{amsmath,amsthm,amssymb,amscd}
\usepackage[all,cmtip,color]{xy}
\usepackage[german,english]{babel}
\usepackage{amsmath}
\usepackage{amssymb}
\usepackage{mathrsfs} 
\usepackage{bbm}
\usepackage{color}
\usepackage{dsfont}
\usepackage{todonotes}

\usepackage[all,cmtip]{xy}
\usepackage[margin = 2.8cm,headsep=1cm]{geometry}
\usepackage{colonequals,enumerate}

\DeclareMathOperator{\Hom}{\mathsf{Hom}}

\DeclareMathOperator{\ev}{\mathrm{ev}}

\DeclareMathOperator{\sm}{sm}

\DeclareMathOperator{\ac}{ac}
\def\llp{\mathopen{(\!(}}
\def\llb{\mathopen{[\![}}

\def\rrp{\mathopen{)\!)}}
\def\rrb{\mathopen{]\!]}}

\DeclareMathOperator{\ooe}{\overline{e}}

\DeclareMathOperator{\A}{\mathsf{A}}

\DeclareMathOperator{\Oc}{\mathcal{O}}
\DeclareMathOperator{\Es}{\mathsf{E}}

\renewcommand{\P}{\mathsf{P}}

\renewcommand{\lim}{\mathsf{lim}}

\newcommand{\wh}{\widehat}

\DeclareMathOperator{\id}{id}

\DeclareMathOperator{\op}{\mathsf{op}}

\DeclareMathOperator{\Gal}{Gal}
\DeclareMathOperator{\Aut}{\mathsf{Aut}}

\DeclareMathOperator{\PGL}{PGL}
\DeclareMathOperator{\SL}{SL}
\DeclareMathOperator{\M}{M}

\DeclareMathOperator{\Spec}{\mathsf{Spec}}

\DeclareMathOperator{\ord}{ord}

\newcommand{\BA}{{\mathbb{A}}}

\newcommand{\BC}{{\mathbb{C}}}

\newcommand{\BF}{{\mathbb{F}}}

\newcommand{\BL}{{\mathbb{L}}}
\newcommand{\BM}{{\mathbb{M}}}

\newcommand{\BP}{{\mathbb{P}}}

\newcommand{\BX}{{\mathbb{X}}}

\newcommand{\BZ}{{\mathbb{Z}}}

\newcommand{\FA}{{\mathcal A}}

\newcommand{\FC}{{\mathcal C}}

\newcommand{\FM}{{\mathcal M}}
\newcommand{\FN}{{\mathcal N}}
 
\newcommand{\FP}{{\mathcal P}}

\DeclareMathOperator{\DA}{DA^{\acute{e}t}}
\DeclareMathOperator{\DAC}{DA^{\acute{e}t}_{ct}}
\newcommand{\QQ}{\mathbb{Q}}
\newcommand{\Var}{\mathrm{Var}}
\newcommand{\cM}{\mathcal{M}}
\newcommand{\cC}{\mathcal{C}}

\newcommand{\dpl}{\mathcal{L}_{\mathrm{DP}}}
\newcommand{\dplk}{\mathcal{L}_{\mathrm{DP},k}}
\newcommand{\dplkx}{\mathcal{L}_{\mathrm{DP},k(x)}}
\newcommand{\Def}{\mathrm{Def}}
\newcommand{\GDef}{\mathrm{GDef}}
\newcommand{\RDef}{\mathrm{RDef}}
\newcommand{\acf}{\mathrm{acf}}
\newcommand{\mot}{\mathrm{mot}}

\newcommand{\im}{\mathrm{Im}}
\newcommand{\Pic}{\mathrm{Pic}}

\newcommand{\FPh}{\widehat{\FP}}

\newcommand{\tr}{\mathrm{tr}}
\newcommand{\Nm}{\mathrm{Nm}}
\newcommand{\Stab}{\mathrm{Stab}}
\renewcommand{\M}{\mathrm{M}}
\newcommand{\st}{\mathrm{st}}
\renewcommand{\op}{\overline{\partial}}
\newcommand{\bs}{\boldsymbol}

\newtheorem{theorem}[subsubsection]{Theorem}
\newtheorem{proposition}[subsubsection]{Proposition}
\newtheorem{aproposition}[subsection]{Proposition}
\newtheorem{propdef}[subsubsection]{Proposition-Definition}
\newtheorem{corollary}[subsubsection]{Corollary}

\newtheorem{lemma}[subsubsection]{Lemma}

\theoremstyle{definition}
\newtheorem{definition}[subsubsection]{Definition}
\newtheorem{construction}[subsubsection]{Construction}
\newtheorem{rmk}[subsubsection]{Remark}
\newtheorem{rmks}[subsubsection]{Remarks}

\numberwithin{equation}{subsection}

\begin{document}

\author{Fran\c {c}ois Loeser}
\address{Institut universitaire de France, Sorbonne Universit\'e, Institut de Math\'ematiques de Jussieu-Paris
Rive Gauche, CNRS. Univ Paris Diderot, Campus Pierre et Marie Curie, case 247, 4 place Jussieu, 75252 Paris cedex 5, France.
}
\email{francois.loeser@imj-prg.fr}
\urladdr{https://webusers.imj-prg.fr/$\sim$francois.loeser/}
\author{Dimitri Wyss}
\address{EPFL/SB/ARG, Station 8, CH-1015 Lausanne, Switzerland
}
\email{dimitri.wyss@epfl.ch}
\urladdr{https://people.epfl.ch/dimitri.wyss}

\title{Motivic integration on the Hitchin fibration}

\maketitle 
\begin{abstract} We prove that the moduli spaces of twisted $\SL_n$ and $\PGL_n$-Higgs bundles on a smooth projective curve have the same (stringy) class in the Grothendieck ring of rational Chow motives. On the level of Hodge numbers this was conjectured by Hausel and Thaddeus, and recently proven by Groechenig, Ziegler and the second author. To adapt their argument, which relies on $p$-adic integration, we use 
a version of motivic integration with values in rational Chow motives  %as developed by Cluckers and the first author
and  the geometry of N\'eron models to evaluate such integrals on Hitchin fibers.
\end{abstract}

\tableofcontents

\section{Introduction}The main goal of the present work is to provide a proof of a motivic version of the topological mirror symmetry conjecture of
Hausel and Thaddeus \cite{MR1990670}.

\subsection{Reminders on the topological mirror symmetry conjecture}
Let $C$ be a connected, smooth and projective  complex curve  and let $G$ be a Lie group. 
A $G$-Higgs bundle on $C$  is a pair $(E, \theta)$ with $E$ a principal $G$-bundle on $C$
and 
$\theta$ an element of $H^0 (C, \mathrm{ad} \, E \otimes K_C)$,
with  
$\mathrm{ad} \, E = E \otimes_G \mathfrak{g}$ the adjoint vector bundle of $G$ and $K_C$ the canonical bundle of $C$.
By assigning  to $(E, \theta)$ the 
characteristic polynomial of $\theta$, one defines the Hitchin fibration on the corresponding moduli space.
Hausel and Thaddeus conjectured in \cite{MR1990670}  that the moduli spaces of $\mathrm{SL}_n$ and $\PGL_n$-Higgs bundles are mirror partners
in the sense of Strominger-Yau-Zaslow. 
In this case the 
two Hitchin fibrations have the same base and Hausel and Thaddeus proved  that over a dense open set their fibers are torsors under dual abelian varieties.
They also conjectured that the two moduli spaces have the same Hodge numbers, when defined in an
appropriate way.

More precisely, fix a line bundle $L$ of degree $d$ on $C$ with $d$ prime to $n$.
We consider the moduli space $\M_n^L$ of semi-stable $L$-twisted $\mathrm{SL}_n$-Higgs bundles on $C$, which is a smooth quasi-projective variety.
One can also consider
 the moduli space  $\widehat{\M}_n^d$ of semi-stable $\PGL_n$-Higgs bundles of degree $d$  on $C$.
It is an orbifold which can be identified with the geometric quotient of $ \M_n^L$ 
by the natural action of $\Gamma = \Pic(C)[n]$ given by twisting the underlying vector bundle of a Higgs field.

The topological mirror symmetry conjecture of
Hausel and Thaddeus \cite{MR1990670}, now proven by 
Groechenig, Ziegler and the second author in \cite{gwz}, is the following statement (see Theorem \ref{htconj} for a slightly more general statement):

\begin{theorem}\label{tmsintro} Let $d$ be an integer prime to $n$ and $L$ line bundles on $C$ of degree $d$. Then there is an equality of Hodge numbers
\begin{equation}\label{tmsintrof}
h^{p,q} (\M_n^L) = 
h^{p,q}_{\st} (\widehat{\M}_n^{d}, \alpha_{L}^d).
\end{equation}
\end{theorem}

A few words are in order to explain the meaning of the right hand side of this equality. 
Stringy Hodge are invariants introduced by Batyrev \cite{batyrevstringy}\cite{batyrevjems} 
for algebraic varieties with log terminal singularities. They  are  especially useful to extend topological mirror symmetry 
for pairs of singular Calabi-Yau varieties beyond the smooth case. 
The numbers appearing on the right hande side are twisted versions of the 
stringy Hodge numbers of $\widehat{\M}_n^{d}$. The twisting involves a $\mu_n$-gerbe
$\alpha_{L}$ on $\widehat{\M}_n^{d}$ that roughly speaking describes
the  duality between generic Hitchin fibers, which are torsors under dual abelian varieties. Concretely $\alpha_{L}$ has the effect of singling out certain isotypical components as we explain next.

\subsection{Hodge-Deligne polynomials} In this paragraph we assume $k=\BC$. Let  $M$ be a smooth connected variety over $k$ with an action of a finite abelian group $\Gamma$ preserving its canonical bundle.
For  $\gamma \in \Gamma$ and  $x$ a point in $M$ fixed by $\gamma$, one defines an integer
$w_x(\gamma)$ in terms of the eigenvalues of $\gamma$ acting on the tangent space at $x$, cf. \ref{weight}. In the cases we will consider 
$w_x (\gamma)$ is constant on the fixed point set $M^\gamma$ for each $\gamma \in \Gamma$. Under this assumption, the stringy $E$-polynomial of the quotient $M/\Gamma$ is given by
\[ E_{\st}(M /\Gamma;x,y) = \sum_{\gamma \in \Gamma} (xy)^{\dim M-w(\gamma)} E(M^\gamma/\Gamma;x,y),\]
with $E$ the usual Hodge-Deligne polynomial.

We now specialize to the case  when
 $M = \M_n^L$ and $\Gamma = \Pic^0(C)[n]$. Let
$\varrho= \langle \cdot,\cdot \rangle: \Gamma \times \Gamma \to \mu_n$
be the Weil pairing on $\Gamma = \Pic^0(C)[n]$ and $\varrho_\gamma = \langle \gamma, \cdot \rangle : \Gamma \to \mu_n$ the character induced by $\gamma \in \Gamma$. For any integer $s$ we obtain a twisted stringy $E$-polynomial
\begin{equation}\label{twistedE} E_{\st}^{\varrho^{s}}(\widehat{\M}_n^d;x,y) =
\sum_{\gamma \in \Gamma} (xy)^{\dim \M_n^L -w(\gamma)} E^{\varrho_{\gamma}^{s}} (\M_n^{L, \gamma};x,y),
\end{equation}
with $E^{\varrho_{\gamma}^{s}} (\M_n^{L, \gamma};x,y)$ the $E$-polynomial of the $\varrho_{\gamma}^{s}$-isotypical component of the cohomology with compact supports of $\M_n^{L, \gamma}$.
As recalled in Section \ref{sec5},
(\ref{tmsintrof}) can be restated as the following identity between $E$-polynomials:
\begin{equation}\label{tmsintrofbis}
E (\M_n^L; x, y)
= E_{\st}^{\varrho^{-1}}(\widehat{\M}_n^{d};x,y).
\end{equation}

\subsection{The main result}
Let $r$ be  the order of $\Pic(C)[n]$ and set $\Lambda = \mathbb{Q} (\mu_r)$.
We denote by
$\mathrm{M_{rat}} (k, \Lambda)$ the category of Chow motives over $k$ with coefficients in $\Lambda$, and by
$K_0 (\mathrm{M_{rat}} (k, \Lambda))$ its Grothendieck ring. 
We denote by $[\M_n^L]$ the virtual Chow motive of $\M_n^L$ in $K_0 (\mathrm{M_{rat}} (k, \Lambda))$.
One may also define by a formula similar to (\ref{twistedE}) a virtual Chow motive
$[\widehat{\M}_n^d]_{\st}^{\varrho^{s}} \in K_0 (\mathrm{M_{rat}} (k, \Lambda))$.

We set 
\[K_0 (\mathrm{M_{rat}} (k, \Lambda))_{\mathrm{loc}} :=
K_0 (\mathrm{M_{rat}} (k, \Lambda)) \otimes_{\mathbb{Z} [\mathbb{L}, \mathbb{L}^{-1}]}    \mathbb{A},\]
with $\mathbb{A} = \mathbb{Z} \Bigl[\mathbb{L}, \mathbb{L}^{-1}, \Bigl(\frac{1}{1- \mathbb{L}^{-i}}\Bigr)_{i >0}\Bigr]$,
where $\mathbb{L}$ denotes the class of the Lefschetz motive,
and we denote by 
$\vartheta: K_0 (\mathrm{M_{rat}} (k, \Lambda)) \to K_0 (\mathrm{M_{rat}} (k, \Lambda))_{\mathrm{loc}}$ the localization morphism.

We can now state the main result of this paper:

\begin{theorem}\label{mainintro}
The equality
\begin{equation}\label{mainintrof}\vartheta ([\M_n^L])
=
\vartheta ([\widehat{\M}_n^{d'}]_{\st}^{\varrho^{-1}})
\end{equation}
holds in $K_0 (\mathrm{M_{rat}} (k, \Lambda))_{\mathrm{loc}}$.
\end{theorem}

Equality (\ref{tmsintrofbis}) follows from (\ref{mainintrof}) when applying the $E$-polynomial to both sides.

\subsection{Strategy of proof}The strategy of the proof of (\ref{tmsintrof}) in \cite{gwz} is very roughly the following. 
After spreading out one may assume all the data to be defined over of  subalgebra $R$ of finite type over $\mathbb{Z}$.
Using $p$-adic Hodge theory,  one reduces the proof of    (\ref{tmsintrof}) to proving that, 
for every ring morphism
$R \to \mathbb{F}_q$,
\[\# \M_n^L (\mathbb{F}_q)= 
\#_{\st} (\widehat{\M}_n^d, \alpha_L) (\mathbb{F}_q),\]
where on the right hand side the stringy number of points  $\#_{\st}$ is defined in a  way much similar to the stringy Hodge polynomial, and the twisting by $\alpha_L$ has the effect of replacing
expressions involving number of points  by character sums.
The next step is to move from an equality between number of points over finite fields to an equality between $p$-adic integrals. That is for a non-archimedean local field $F$ with ring of integers $\Oc_F$ and residue field $\BF_q$ one has to prove 
\begin{equation}\label{padic}
\int_{\M_n^L (\Oc_F)} d\mu
=
\int_{\widehat{\M}_n^d (\Oc_F)} f_{\alpha_L^d} d\mu_{orb}. 
\end{equation}
Here $d\mu$ is the canonical $p$-adic measure on $\M_n^L (\Oc_F)$, 
 $d\mu_{orb}$ the corresponding orbifold measure on
$\widehat{\M}_n^d (\Oc_F)$ and $f_{\alpha_L^d}$ a certain function assigned to $\alpha^d_L$.
That this reduction is possible is due to the fact that  the $p$-adic volume of an orbifold over the ring of integers of a non-archimedean local field can be expressed in terms of
stringy point-counting over the residue field. In fact, similar statements for motivic  volumes  can already be found in \cite{DL2002}\cite{yasuda1}\cite{yasuda2} in connection with the McKay correspondence. The proof of (\ref{padic}) proceeds by using Fubini for the two Hitchin fibrations.
Outside a bad locus  of measure zero in the common base of the two Hitchin fibrations, the fibers have a good behaviour, in particular they are torsors under dual abelian varieties. By Fubini it is thus enough to 
prove that for any point $a$ in the base which does not belong to the bad locus, the two fiber integrals over $a$ with respect to  the relative measures  are equal.
This is where Tate Duality for abelian varieties over non-archimedean local fields enters the game.
For instance, when the fiber of $\M_n^L (\Oc_F)$ over $a$ has no rational point, in which case the corresponding fiber integral is zero, the authors of \cite{gwz} are able to show using 
  Tate Duality that the function $f_{\alpha_L^d}$ behaves like a non-trivial character on the fiber over
  $\widehat{\M}_n^d (\Oc_F)$, which implies that the second fiber integral is also zero.

\medskip
In the present paper we use a 
variant of the theory of motivic integration developed by Cluckers and the first author in  \cite{CL-2008}, but for functions 
with values in a Grothendieck rings of  Chow motives with coefficients in a characteristic zero field $\Lambda$ containing enough roots of unity.
It satisfies a Fubini theorem, as does the original theory, which is crucial for our needs.
One of the important features of working  with  Chow motives is that 
in the presence of the action of finite group action, Chow motives may decomposed 
into isotypical components. They can be viewed as motivic analogues of 
character sums over finite fields.
This analogy works quite well,  for instance there is a motivic version of the fact that non-trivial characters sums on commutative algebraic groups are zero. 

\medskip

Our strategy for proving our main result follows the main lines of \cite{gwz}, replacing $p$-adic integration by motivic integration.
We start by expressing  the equality  we want to prove as an equality between motivic integrals on moduli spaces of Higgs bundles over $k \llb t \rrb$. In order to extract twisted stringy invariants from the motivic volume of $\widehat{\M}_n^d$ we prove a equivariant volume formula for orbifolds, Theorem \ref{equivorb}, which might be of independent interest.
Applying Fubini to the corresponding Hitchin fibrations one then  reduces the proof of the equality to comparing fiber integrals outside a bad locus  of measure zero in the common base of the  Hitchin fibrations. As Tate duality is not available in our context, we then argue directly with the Weil pairing on $\Gamma$ and its interaction with N\'eron models 
of generic Hitchin fibers.

\subsection{Organization of the paper}The paper is organized as follows. 
Section \ref{sec2} is devoted to the development of the Chow motive  variant of the theory in  \cite{CL-2008}  and contains also some reminders and complements about Chow motives and motivic integration. 
In Section \ref{sec3} we  study motivic volumes of orbifolds over $k \llb t \rrb$  in this framework. We show in particular that
these can be computed on the coarse moduli space of the inertia stack
of the special fiber. This can be seen as an extension of previous work in \cite{DL2002}\cite{yasuda1}\cite{yasuda2}.
Section \ref{sec4} is devoted  to N\'eron models  and their relation to pairings arising from self-dual isogenies.
%The fact that N\'eron models are well suited for motivic integration was already observed in \cite{LS}.
Our main result,
Theorem \ref{premainthm}, is presented in  Section \ref{sec5}.
We  then tie all the preliminary results  together in Section \ref{sec6}, where we complete the proof  of  Theorem \ref{premainthm}.

\bigskip

\subsection*{Acknowledgements} We thank warmly Joseph Ayoub for his keenness in  answering on the spot our naive questions on \'etale motives and Tamas Hausel for pointing out a mistake in an earlier version. The second author is grateful to Michael Groechenig and Paul Ziegler, as he learned many of the ideas used here from the collaborations with them. 
We are also grateful  to the referee for a  careful reading of the paper and many remarks and comments  that were very helpful in the preparation of the final version.

During the preparation of this work the authors were partially supported  by the ANR grant ANR-15-CE40-0008 (D\'efig\'eo).
F.L. was  also supported by 
the Institut Universitaire de France and D.W. by the Fondation Sciences Math\'ematiques de Paris,
under  ANR-10-LABX0098.

%\cite{batyrev1999birational}

%\cite{DLjams}

%\cite{yasuda1}
%\cite{yasuda2}

%\cite{batyrev1999birational}

%These are numerical invariants introduced by st
%Batyrev [Bat99b], which include appropriate correction terms to compensate for the presence of singular- ities. In addition, the gerbe α living on these spaces needs to be taken into account. This is natural from the point of view of duality of the Hitchin fibers. The proof of this result proceeds by proving an equality for stringy point-counts over finite fields first, by means of p-adic integration. We then use p-adic Hodge

\section{Reminders and complements on motivic integration}\label{sec2}

\subsection{\'Etale motives and their Grothendieck rings}
Let $k$ be a field of characteristic zero and 
$S$ a quasi-projective $k$-scheme. Fix a field $\Lambda$ of characteristic zero.
We will work with the
triangulated category $\DA (S, \Lambda)$ of \'etale motives  over $S$ with coefficients in $\Lambda$
introduced by J. Ayoub. It was first introduced  in \cite[D\'ef. 4.5.21]{Ayoub_ast2} (under a different naming)
and its construction is recalled at the beginning of Section 3 of \cite{Ayoub_ENS}.
By \cite[Prop. 3.2]{Ayoub_ENS} and
\cite[Scholie 1.4.2]{Ayoub_ast1}, it is endowed 
with a Grothendieck
six operations formalism. In the paper  \cite{Ayoub_ENS} J. Ayoub 
also introduced 
the full triangulated subcategory $\DAC (S, \Lambda)$
of constructible \'etale motives for which he proved stability under the six operations
 in Th\'eor\`eme 8.10 and Th\'eor\`eme 8.12 of \cite{Ayoub_ENS}.
We also refer to 
\cite{ayoub_icm} for an accessible and  useful  introduction.

We denote by $K_0 (\Var_S)$ the Grothendieck ring of varieties over $S$
and by $\cM_S$ its localization by the  class $\mathbb{L}$ of the affine line over $S$.

In   \cite[Lemma 2.1]{ivo_seb}, Ivorra and Sebag prove
the existence of a unique ring morphism
$$
\chi_{S, c} :
\cM_S \longrightarrow K_0(\DAC (S, \Lambda))
 $$
which assigns to a quasi-projective $S$-scheme $p : X \to S$ the element $\chi_{S, c} ([X]) := [p_! (\mathds{1}_X)]$ in
$K_0( \DAC (S, \Lambda))$.

Now, if $q: S \to T$ is a morphism of quasi-projective $k$-schemes,
the functors
$$q_!: \DAC (S, \Lambda) \to \DAC (T, \Lambda)$$
and
$$q^*: \DAC (T, \Lambda) \to \DAC (S, \Lambda)$$
induce morphisms of Grothendieck rings
$$q_!: K_0( \DAC (S, \Lambda)) \to K_0( \DAC (T, \Lambda))$$
and
$$q^*: K_0( \DAC (T, \Lambda)) \to K_0( \DAC (S, \Lambda)).$$

Composition with $q$ and pullback induce respectively morphisms 
$$q_! : \cM_S \to \cM_T$$
and
$$q^* : \cM_T \to \cM_S.$$

It follows directly from the constructions and Ayoub's six operations formalism for the category
$\DAC (S, \Lambda)$ of constructible \'etale motives as recalled above, 
that
the diagrams 
 \begin{equation}\label{eq1}\xymatrix{  \cM_S \ar[d]^{q_!} \ar[r]^-{\chi_{S, c}} &  K_0( \DAC (S, \Lambda))  \ar[d]^{q_!}  \\
							 \cM_T \ar[r]^-{\chi_{T, c}} & K_0( \DAC (T, \Lambda)) 
}\end{equation}
and
\begin{equation}\label{eq2} \xymatrix{  \cM_T \ar[d]^{q^*} \ar[r]^-{\chi_{T, c}} &  K_0( \DAC (T, \Lambda))  \ar[d]^{q^*}  \\
							 \cM_S \ar[r]^-{\chi_{S, c}} & K_0( \DAC (S, \Lambda)) 
}\end{equation}
commute.

\begin{lemma}\label{evmot}Let $S$ be a quasi-projective $k$-scheme.
Let $\beta \in K_0 (\DAC (S, \Lambda))$ such  that, for every point $i_x : x \hookrightarrow S$,
$i_x^{*} (\beta) = 0$. Then $\beta = 0$.
\end{lemma}
\begin{proof}Let   $i_x : x \hookrightarrow S$  be a  a point of $S$. As an
$S$-scheme, it  is the limit of the family of locally closed subsets $C$ of $S$ containing $x$.
Since the family of locally closed subsets of the form $F \cap U$ with $F$ closed and $U$ open affine is cofinal in the family of all such $C$'s,
it follows from  \cite[Cor. 3.22]{Ayoub_ENS}
that
the category
$\DAC (x, \Lambda)$ is equivalent to the $2$-colimit of the categories
$\DAC (C, \Lambda)$. In particular, the ring
$K_0 (\DAC (x, \Lambda))$ is the colimit of the rings
$K_0 (\DAC (C, \Lambda))$
for $C$ running over the locally closed subschemes of $S$ containing  $x$.
Assume now that,  for every point $i_x : x \hookrightarrow S$,
$i_x^{*} (\beta) = 0$. It follows that there exists a cover of
$S$ by locally closed subschemes $C$ such that $i_C^*(\beta)=0$ for every $C$.
We may assume that this cover is finite and that the locally closed subschemes $C$ form a partition of $S$.
One concludes by induction on the cardinality of the cover, by using that if
$j:U\to S$ is an open immersion and  $i:Z\to S$ is the inclusion of the complementary closed subscheme, then
$\beta=j_!j^*(\beta)+i_!i^*(\beta)$.
\end{proof}

\begin{rmk}\label{bonda}When $S = \Spec k$, we shall write $\DAC (k, \Lambda)$, $\cM_k$, etc, for the categories
$\DAC (\Spec k, \Lambda)$, $\cM_{\Spec k}$, etc.
The additive category $\mathrm{M_{rat}} (k, \Lambda)$ of Chow motives over $k$ with coefficients  in $\Lambda$ embeds
in 
$\DAC (k, \Lambda)$ (cf. \cite[2.2.5]{ivo_seb}), so we have a natural ring morphism
$\iota: K_0 (\mathrm{M_{rat}} (k, \Lambda)) \to K_0 (\DAC (k, \Lambda))$.
This morphism is an isomorphism 
since the category $\DAC (k, \Lambda))$
is equivalent to Voevodsky's category
$\mathrm{DM_{gm}}(k,\Lambda)$ by \cite[Appendice B]{Ayoub_ENS} 
and the ring $K_0 (\mathrm{DM_{gm}}(k,\Lambda))$ is isomorphic to
$K_0 (\mathrm{M_{rat}} (k, \Lambda))$, a statement  proved by 
Bondarko in \cite{bondarko} as a consequence
of his theory of weight structures (in loc. cit. only the case of
$\Lambda = \QQ$ is considered, but the proofs carry over for general $\Lambda$).
One can check that
$\chi_{k, c}: \cM_k \to K_0(\DAC (k, \Lambda))$ is equal to
$\iota \circ \chi_c$ with $\chi_c: \cM_k \to K_0 (\mathrm{M_{rat}} (k, \Lambda))$
the morphism considered in
\cite{DL-JAG} (when $\Lambda = \QQ$).
\end{rmk} 
 
\subsection{Constructible motivic functions}
We shall use in this paper the formalism of constructible motivic functions developed by R. Cluckers and the first author in the paper \cite{CL-2008}.
The introduction to the  paper \cite{chl} may also be useful to some readers. Let us review some of the features we will use in this paper.

\subsubsection{The basic framework}
We fix a field $k$ of characteristic $0$ and we work in the Denef-Pas language $\dplk$.
It is a $3$-sorted language in the sense of first order logic,
the sorts being respectively  the valued
field sort, the residue field sort, and the value group sort.
The language consists of the disjoint union of
the language of rings with coefficients in
$k \llp t \rrp$ restricted to the valued field sort, of
the language of rings with coefficients in
$k$ restricted to the residue field sort and of
the language
of ordered groups restricted to the value group sort, together with
two additional symbols of unary functions {$\ac$}  and {$\ord$} from the valued field sort to the
residue field and valued groups sort, respectively.
Furthermore, for the value group sort we add symbols
symbols  $\equiv_n$, for
 $n >1$ in $\mathbb{N}$.

A typical example of $\dplk$-structure is provided by 
$(k \llp t \rrp, k, \mathbb{Z})$
with $\ac$ interpreted as the function
$\ac : k \llp t \rrp \rightarrow k$ assigning to a series its first nonzero coefficient if not zero, zero otherwise, $\ord$ interpreted as the valuation function
$\ord : k \llp t)\rrp \setminus \{0\} \rightarrow \mathbb{Z}$, and $\equiv_n$ interpreted as equivalence relation modulo $n$.
%(There is a minor divergence here,
%easily fixed,
 %since
%$\ord \, 0$ is not defined.)
More generally, for any field $K$ containing $k$,
$(K \llp t \rrp, K, \mathbb{Z})$ is naturally an $\dplk$-structure.

The constructions of the paper \cite{CL-2008} take place in  a category 
$\Def_k$, also written $\Def_k(\dpl)$, of definable objects in the language $\dplk$. Objects of $\Def_k$ are called definable subassignments (or definable sets).
They are defined by formulas in $\dplk$ as follows.
Let $\varphi$ be a formula in the language $\dplk$  having
respectively $m$, $n$, and $r$
free variables in the various sorts. To such  a formula $\varphi$
we assign, for every field $K$ containing $k$, the subset $h_{\varphi} (K)$
of $K \llp t \rrp^m \times K^n \times \mathbb{Z}^r$
consisting of all points satisfying $\varphi$.
An object $S$ of $\Def_k$ consists of the
datum of such subsets $h_{\varphi} (K)$ for all $K$ for some $\varphi$.
In particular we may set $S (K) = h_{\varphi} (K)$.
There is also a global variant $\GDef_k$ of $\Def_k$, whose objects are definable subassignments of algebraic varieties defined over $k$, which is defined similarly via affine charts.
For any    $S$ in  $\Def_k$ or  in $\GDef_k$,   a ring $\cC (S)$ of constructible motivic functions is constructed in \cite{CL-2008} and the main achievement of
that paper is the construction of a 
theory of integration for such functions.

\subsubsection{The framework we shall use}
In fact, we will not deal directly with $\Def_k(\dpl)$ in the present paper, but instead we will 
work with  a variant considered in 
\cite[16.2]{CL-2008}. Namely, let
$T_\acf$ be the theory of algebraically closed fields containing $k$,
we shall work in the category
denoted by 
$\Def_k(\dpl,T_\acf)$ in loc. cit.
Concretely, an object in $\Def_k(\dpl,T_\acf)$ is obtained  by evaluating an object
$\Def_k$ only at algebraically closed fields containing $k$:
an object $S$ of $\Def_k(\dpl,T_\acf)$ consists of the
datum of the subsets $S (K) = h_{\varphi} (K)$ for all $K$ algebraically closed containing $k$, for some formula $\varphi$ as above.
By quantifier elimination for $T_\acf$ in $\dplk$,  for every such $S$ there exists a formula $\varphi$ which is quantifier free.
This feature shows an important difference between
$\Def_k(\dpl,T_\acf)$ and $\Def_k$, since for objects in the later category, it is in general not possible to have quantifier free formulas in the residue field variables.
The present paper being geometric in nature we can work in
$\Def_k(\dpl,T_\acf)$, but for more arithmetical questions one would have to stick to $\Def_k$.
To shorten notation, we shall write $\Def_{\acf,k}$ instead  of $\Def_k(\dpl,T_\acf)$ and  $\GDef_{\acf,k}$ instead of $\GDef_k(\dpl,T_\acf)$.
Similarly, if $S$ is an object of $\Def_{\acf,k}$, we shall write $\cC_\acf (S)$
instead of $\cC (S, (\dpl,T_\acf))$, etc.

To  any algebraic variety $X$ over $k\llp t \rrp$ corresponds
the object $\underline{X}$ in $\GDef_{\acf,k}$
defined by 
$\underline{X} (K) = X(K\llp t \rrp)$ and $X \mapsto \underline{X}$ is a functor. 
%By abuse of notation, if $X$ is defined over $k\llb t\rrb$, we will also write $\underline{X}$ 
Similarly, there is a functor
$X \mapsto \underline{X}_{{}}$ from
algebraic varieties over $k$ to
$\GDef_{\acf,k}$,
defined by  
$\underline{X}_{{}} (K) = X(K)$. Throughout the paper it should always be clear whether a given algebraic variety is defined over 
$k\llp t \rrp$ or $k$, to avoid any risk of confusion.
Finally,
if $X$ is defined over $k\llb t\rrb$, we will  write $\underline{X}_{\circ}$ 
for the assignment $\underline{X}_{\circ}(K) = X(K\llb t\rrb)$.

\subsubsection{Evaluation of functions}
Let  $S$ be in $\GDef_{\acf,k}$ and let $K$ be an algebraically closed field containing
$k$. For $x$ in $S (K)$, we denote by $k(x)$ the field  generated by $k$, by the coefficients of the
valued field components of $x$ and by the residue field components of $x$ (in the global case one reduces to the affine case).
Note that for any algebraically closed field $K'$ containing $k (x)$, $x \in S (K')$ by quantifier elimination, and that the field
$k (x)$ is independent from the choice of $K$.
We fix a Grothendieck universe $\mathcal{U}$ containing $k$ and we 
define the set of points
$\vert S \vert$ as the colimit of the sets $S (K)$ where $K$ belongs to the category of fields extensions of $k$ in
 $\mathcal{U}$.
[Note that this definition is slightly different from the one given in \cite{CL-2008}.]

Let $\ast_{k}$ be the terminal object in $\GDef_{\acf,k}$.
Its set of points $\vert \ast_{k} \vert$ consists of a unique point $\ast_{k}$ with
$k (\ast_{k}) = k$.
By construction
$$
\cC_\acf (\ast_{k}) = \cM_k \otimes_{\mathbb{Z} [\mathbb{L}, \mathbb{L}^{-1}]}    \mathbb{A} = K_0 (\Var_k)  \otimes_{\mathbb{Z} [\mathbb{L}]} \mathbb{A}
$$
with $\mathbb{A} = \mathbb{Z} \Bigl[\mathbb{L}, \mathbb{L}^{-1}, \Bigl(\frac{1}{1- \mathbb{L}^{-i}}\Bigr)_{i >0}\Bigr]$.

For a general $S$ in $\GDef_{\acf,k}$, one can evaluate a function   $\varphi \in \cC_\acf (S)$ at $x \in \vert S \vert$ as follows.
We have an inclusion morphism
$i_x : x \hookrightarrow S \otimes k(x)$ in $\GDef_{\acf,k (x)}$, where $S \otimes k (x)$ is obtained from $S$ by extension of scalars
and we identify $x$ with the terminal object in $\GDef_{\acf,k (x)}$.
The function $\varphi$ gives rise by base change to a function in $\cC_\acf (S\otimes k(x))$
and one defines 
$\varphi (x)$ as its pullback  under the inclusion morphism $i_x$.
It is an element of
$ \cM_{k(x)} \otimes_{\mathbb{Z} [\mathbb{L}, \mathbb{L}^{-1}]}  \mathbb{A}$.
It is proved in  \cite{eval} that 
$\varphi = 0$ if and only $\varphi (x) = 0$ for every point $x$ in  $\vert S \vert$.

 \subsubsection{A Fubini Theorem}\label{ssf}

Let $\varphi$ be in $\cC_\acf (S)$ with $S$ in $\Def_{\acf,k}$ of $K$-dimension $d$ in the terminology of \cite{CL-2008}.
We denote by $\vert \omega_0\vert_S$ the canonical volume form in the sense of \cite[15.1]{CL-2008}.
As in  \cite[6.2]{CL-2008}  denote by  $C_\acf^d (S)$ the quotient of $\cC_\acf(S)$ by the ideal of functions with support contained in a definable subset of $K$-dimension at most $d-1$.
When the class $[\varphi]$ of $\varphi$ in $C^d_\acf (S)$ is integrable, the integral $\int_S [\varphi] \, \vert \omega_0\vert_S$
was defined in \cite[15.1]{CL-2008}. It belongs to $\cC_\acf (\ast_{k})$.
We shall say $\varphi$ is integrable if its class $[\varphi]$ in $C_\acf^d (S)$ is and we shall  set 
\[\int_S \varphi \, \vert \omega_0\vert_S := \int_S [\varphi] \, \vert \omega_0\vert_S.\]
Note that if $\varphi$ is zero outside
a definable subset of $K$-dimension at most $d-1$, then $\int_S \varphi \, \vert \omega_0\vert_S = 0$.

Similarly, let $X$ be a 
smooth algebraic variety over $k\llp t \rrp$ of pure dimension $d$.  Let $\omega_X$ be a degree $d$ algebraic differential form on $X$, and
$\varphi \in \cC_\acf (\underline{X})$.
We shall say $\varphi \, \vert \omega_X \vert$ is integrable if
$[\varphi] \, \vert \omega_X \vert$ is integrable and we set
$ \int_{\underline{X}} \varphi \, \vert \omega_X \vert := \int_{\underline{X}} [\varphi] \, \vert \omega_X \vert$.

We shall use the following statement, which follows directly from Theorem 10.1.1, Proposition 15.4.1 and Section 16 of \cite{CL-2008}.

\begin{proposition}\label{fubini}\begin{enumerate}
\item Let $X$ and $Y$ be 
smooth algebraic varieties over $k\llp t \rrp$ of pure dimension. 
Let $f : X \to Y$ be a smooth morphism. %We denote by  $\mathcal{X}$ and $\mathcal{Y}$ the objects of $\GDef_{\acf,k}$
We denote by  $\underline{f} : \underline{X} \to \underline{Y}$ the morphism induced by $f$.
Let $\omega_X$ and $\omega_Y$ be top degree algebraic differential forms on $X$ and $Y$, respectively, everywhere non-zero.
Let $\varphi \in \cC_\acf (\underline{X})$ such that $\varphi \, \vert \omega_X \vert$ is integrable.
For any point
$y$ in $\underline{Y}$, denote by 
$\underline{X}_y$ the fiber of $\underline{f}$ at $y$.
Then there exists a definable subset $Z$ of $\underline{Y}$ of $K$-dimension at most $\dim Y - 1$
and
a constructible motivic function
$\psi \in \cC_\acf (\underline{Y})$ with $\psi \, \vert \omega_Y \vert$ integrable
such that, for any point $y$ in $\underline{Y} \setminus Z$,
$\varphi_{\vert \underline{X}_y} \, \vert \omega_X / f^* (\omega_Y) \vert_{\vert X_y}$ is integrable on
$\underline{X}_y$,
$$
\psi (y) = \int_{\underline{X}_y} \varphi_{\vert \underline{X}_y} \, \vert \omega_X / f^* (\omega_Y) \vert_{\vert \underline{X}_y},
$$
and
$$
\int_{\underline{X}} \varphi \, \vert \omega_X \vert =
\int_{\underline{Y}} \psi \, \vert \omega_Y \vert.
$$
\item Let $X$  be a
smooth algebraic variety over $k\llp t \rrp$ of pure dimension. Let $Y$ be an algebraic variety over $k$.
Let  $\underline{f} : \underline{X} \to \underline{Y}_{{}}$ be a morphism in $\GDef_{\acf,k}$.
Let $\omega_X$ be a top degree form on $X$ which is everywhere non-zero.
Let $\varphi \in \cC_\acf (\underline{X})$ such that $\varphi \, \vert \omega_X \vert$ is integrable.
Then, for any point
$y$ in $\underline{Y}_{{}}$,  $\varphi_{\vert \underline{X}_y} \, \vert \omega_X  \vert_{\vert X_y}$ is integrable on
$\underline{X}_y$
and there exists a constructible motivic function
$\psi \in \cC_\acf (\underline{Y}_{{}})$ such that,
for every point $y$  of $\underline{Y}_{{}}$,
$$
\psi (y) = \int_{\underline{X}_y} \varphi_{\vert \underline{X}_y} \, \vert \omega_X  \vert_{\vert \underline{X}_y},
$$
and
$$
\int_{\underline{X}} \varphi \, \vert \omega_X \vert =
\int_{\underline{Y}_{{}}} \psi.
$$
\end{enumerate}
\end{proposition}

\subsection{Motivic functions with values in motives}
For  $S$ in $\GDef_{\acf,k}$, we have a natural morphism
$$
\vartheta_S:  \cC_\acf (S) \longrightarrow \prod_{x \in \vert S \vert} K_0( \DAC (k(x), \Lambda)) \otimes_{\mathbb{Z} [\mathbb{L}, \mathbb{L}^{-1}]}  \mathbb{A}
$$
sending
a motivic function $\varphi$ to $(\chi_{k(x), c} (\varphi (x)))$.
We set  $\cC_\mot (S) := \vartheta_S (\cC_\acf (S))$ and still denote by
$\vartheta_S$ the induced morphism
$\vartheta_S: \cC_\acf (S) \to \cC_\mot (S)$.

Let now $X$ be a quasi-projective  variety over $k$.
In this case, $\cC_\acf (\underline{X}_{{}})$ can be canonically identified with
$\cM_X \otimes_{\mathbb{Z} [\mathbb{L}, \mathbb{L}^{-1}]}    \mathbb{A}$.
Consider the morphism
$$
\chi_{X, c} \otimes \mathbb{A} : \cM_X \otimes_{\mathbb{Z} [\mathbb{L}, \mathbb{L}^{-1}]}    \mathbb{A}
\longrightarrow K_0( \DAC (X, \Lambda)) \otimes_{\mathbb{Z} [\mathbb{L}, \mathbb{L}^{-1}]}    \mathbb{A}.
$$
It follows from Lemma \ref{evmot}
 that $\cC_\mot (\underline{X}_{{}})$ can be canonically identified with
the image of $\chi_{X, c} \otimes \mathbb{A}$ and that under that identification
the morphisms $\vartheta_{\underline{X}_{{}}}:\cC_\acf (\underline{X}_{{}}) \to \cC_\mot (\underline{X}_{{}})$ and $\chi_{X, c} \otimes \mathbb{A}: \cC_\acf (\underline{X}_{{}}) \to \cC_\mot (\underline{X}_{{}})$ are equal.

Consider a morphism $q: X \to Y$ between quasi-projective  varieties over $k$. Tensoring with $\mathbb{A}$ the morphisms $q_!$ and $q^*$
one gets morphisms
$q_! : \cC_\acf (\underline{X}_{{}}) \to \cC_\acf (\underline{Y}_{{}})$
and
$q^* : \cC_\acf (\underline{Y}_{{}}) \to \cC_\acf (\underline{X}_{{}})$.
One deduces
from the commutative diagrams
(\ref{eq1}) and (\ref{eq2}) that the morphisms 
$q_!: \cM_X \otimes_{\mathbb{Z} [\mathbb{L}, \mathbb{L}^{-1}]}    \mathbb{A}
\to 
\cM_Y \otimes_{\mathbb{Z} [\mathbb{L}, \mathbb{L}^{-1}]}    \mathbb{A}$
and
$q^*: \cM_Y \otimes_{\mathbb{Z} [\mathbb{L}, \mathbb{L}^{-1}]}    \mathbb{A}
\to 
\cM_X \otimes_{\mathbb{Z} [\mathbb{L}, \mathbb{L}^{-1}]}    \mathbb{A}$
obtained by tensoring with $\mathbb{A}$
induce morphisms
$q_! : \cC_\mot (\underline{X}_{{}}) \to \cC_\mot (\underline{Y}_{{}})$
and
$q^* : \cC_\mot (\underline{Y}_{{}}) \to \cC_\mot (\underline{X}_{{}})$
and that
the diagrams 
 \begin{equation}\label{eq3}\xymatrix{  \cC_\acf (\underline{X}_{{}}) \ar[d]^{q_!} \ar[r]^-{\vartheta_{\underline{X}_{{}}}} &  \cC_\mot (\underline{X}_{{}})  \ar[d]^{q_!}  \\
							 \cC_\acf (\underline{Y}_{{}}) \ar[r]^-{\vartheta_{\underline{Y}_{{}}}} & \cC_\mot (\underline{Y}_{{}}) 
}\end{equation}
and
\begin{equation}\label{eq4} \xymatrix{  \cC_\acf (\underline{Y}_{{}}) \ar[d]^{q^*} \ar[r]^-{\vartheta_{\underline{Y}_{{}}}} &  \cC_\mot (\underline{Y}_{{}})  \ar[d]^{q^*}  \\
							 \cC_\acf (\underline{X}_{{}}) \ar[r]^-{\vartheta_{\underline{X}_{{}}}} & \cC_\mot (\underline{X}_{{}})
}\end{equation}
commute.

A crucial property of $\vartheta_S$ is that it commutes with integration:

\begin{propdef}\label{indrel}Let $S$ be in $\Def_{\acf,k}$.
Let $\varphi$ and $\varphi'$ be in $\cC_\acf (S)$. 
Assume $\varphi$ and $\varphi'$ are integrable
and that $\vartheta_S (\varphi) = \vartheta_S (\varphi')$.
Then
$$
\vartheta_{\ast_{k}} \Bigl(\int_S \varphi \, \vert \omega_0\vert_S\Bigr) =
\vartheta_{\ast_{k}} \Bigl(\int_S \varphi' \, \vert \omega_0\vert_S\Bigr)
$$
in $\cC_\mot ({\ast_{k}})$ (and thus also in $K_0 (\mathrm{M_{rat}} (k, \Lambda)) \otimes_{\mathbb{Z} [\mathbb{L}, \mathbb{L}^{-1}]}    \mathbb{A}$).
In this case we say  $\vartheta_S (\varphi)$ is integrable and we set
$$
\int_S \vartheta_S (\varphi) \, \vert \omega_0\vert_S := 
\vartheta_{\ast_{k}} \Bigl(\int_S \varphi \, \vert \omega_0\vert_S\Bigr).
$$
\end{propdef}
\begin{proof}
The construction of the motivic integral in \cite{CL-2008} is performed in several steps. 
It uses cell decomposition to reduce to one of the following three cases: (i) integration over cells in one variable in the valued field sort in the sense of \cite[7.1]{CL-2008}, 
(ii) integration over residue field variables and (iii) integration over  value group variables.
Integration over cells in one variable in the valued field sort uses pullbacks over residue field variables and   value group variables.
Once one notices that pullback over residue field variables commutes with $\vartheta$ by 
the commutativity of diagram
(\ref{eq4}) and that pullback over value group variables obviously commutes with $\vartheta$, it is clear that (i) 
commutes with $\vartheta$.
Commutation of (ii) with $\vartheta$
follows directly from the commutativity of diagram
(\ref{eq3}) while commutation of  (iii) with $\vartheta$ is clear.
\end{proof}

Let $X$  be a
smooth algebraic variety over $k\llp t \rrp$ of pure dimension and let $\omega_X$ be a top degree form on $X$.
Let $\varphi \in \cC_\acf (\underline{X})$ such that $\varphi \, \vert \omega_X \vert$ is integrable.
Using affine charts, one deduces from Proposition \ref{indrel}, that $\vartheta_{\ast_{k}} \Bigl(\int_{\underline{X}} \varphi \, \vert \omega_X\vert_S\Bigr)$
depends only on $\vartheta_{\underline{X}}  (\varphi)$. In this case we say  $\vartheta_{\underline{X}}  (\varphi) \, \vert \omega_X \vert$ is integrable and we set
$$
\int_{\underline{X}}^{\mathrm{mot}} \varphi \, \vert \omega_X \vert
:= \vartheta_{\ast_{k}} \Bigl(\int_{\underline{X}} \varphi \, \vert \omega_X\vert_S\Bigr).
$$

Similarly, if $Y$ is an algebraic variety over $k$ and
$\varphi \in \cC_\acf (\underline{Y}_{{}})$, one sets
$$
\int_{\underline{Y}_{{}}}^{\mathrm{mot}} \varphi 
:= \vartheta_{\ast_{k}} \Bigl(\int_{\underline{Y}_{{}}} \varphi \Bigr).
$$

In particular, we have the following consequence of Proposition \ref{fubini}:

\begin{proposition}\label{motfubini}
\begin{enumerate}
\item
Let $X$ and $Y$ be 
smooth algebraic varieties over $k\llp t \rrp$ of pure dimension. 
Let $f : X \to Y$ be a smooth morphism. %We denote  by $\underline{f} : \underline{X} \to \underline{Y}$ the morphism induced by $f$.
Let $\omega_X$ and $\omega_Y$ be top degree forms on $X$ and $Y$, respectively, everywhere non-zero.
Let $\varphi \in \cC_\acf (\underline{X})$ such that $\varphi \, \vert \omega_X \vert$ is integrable.
Then there exists a definable subset $Z$ of $\underline{Y}$ of $K$-dimension at most $\dim Y - 1$
and
$\psi \in \cC_\acf (\underline{Y})$ with $\psi \, \vert \omega_Y \vert$ integrable
such that, for any point $y$ in $\underline{Y} \setminus Z$,
$\vartheta_{\underline{X}_y} (\varphi_{\vert \underline{X}_y}) \, \vert \omega_X / f^* (\omega_Y) \vert_{\vert \underline{X}_y}$ is integrable,
$$
\vartheta_{\underline{Y}} (\psi) (y) = \int_{\underline{X}_y}^{\mathrm{mot}} \varphi_{\vert \underline{X}_y} \, \vert \omega_X / f^* (\omega_Y) \vert_{\vert \underline{X}_y},
$$
and
$$
\int_{\underline{X}}^{\mathrm{mot}} \varphi \, \vert \omega_X \vert =
\int_{\underline{Y}}^{\mathrm{mot}} \psi \, \vert \omega_Y \vert.
$$
\item
Let $X$  be a
smooth algebraic variety over $k\llp t \rrp$ of pure dimension. Let $Y$ be an algebraic variety over $k$.
Let  $\underline{f} : \underline{X} \to \underline{Y}_{{}}$ be a morphism in $\GDef_{\acf,k}$.
Let $\omega_X$ be a top degree form on $X$ which is everywhere non-zero.
Let $\varphi \in \cC_\acf (\underline{X})$ such that $\varphi \, \vert \omega_X \vert$ is integrable.
Then, for any point
$y$ in $\underline{Y}_{{}}$,
$\vartheta_{\underline{X}_y} (\varphi_{\vert \underline{X}_y}) \, \vert \omega_X  \vert_{\vert \underline{X}_y}$ is integrable 
and there exists $\psi \in \cC_\acf (\underline{Y}_{{}})$  such that,
for every point $y$  of $\underline{Y}_{{}}$,
$$
\vartheta_{\underline{Y}_{{}}} (\psi) (y) = \int_{\underline{X}_y}^{\mathrm{mot}} \varphi_{\vert \underline{X}_y} \, \vert \omega_X  \vert_{\vert \underline{X}_y},
$$
and
$$
\int_{\underline{X}}^{\mathrm{mot}} \varphi \, \vert \omega_X \vert =
\int_{\underline{Y}_{{}}}^{\mathrm{mot}} \psi.
$$
\end{enumerate}
\end{proposition}

\subsection{The equivariant case}\label{equicase}

We fix a finite commutative group $\Gamma$. 
%We assume from now on that our base field $k$ contains $\mathbb{C}$.
%We fix a positive integer $n$ and we denote by $\mu_n = \mu_n (\mathbb{C})$ the group of $n$-th roots of unity in $\mathbb{C}$.
Let $S$ be an object of $\GDef_{\acf,k}$. 
Recall that the ring
$\cC_\acf (S)$ of constructible motivic functions is defined in Section 5.3 of 
\cite{CL-2008}
as a tensor product
\[\cC_\acf (S) := K_0 (\RDef_{\acf,S}) \otimes_{\mathcal{P}^0 (S)} \mathcal{P} (S).\]
Here the ring $\mathcal{P} (S)$ is the ring of Presburger functions defined in Section 4.6 of \cite{CL-2008},
$\mathcal{P}^0 (S)$ is a subring of $\mathcal{P} (S)$ defined in Section 5.2 of \cite{CL-2008}, and
$K_0 (\RDef_{\acf,S})$ is the Grothendieck ring of 
the full subcategory
$\RDef_{\acf,S}$   of $\GDef_{\acf,S}$
whose objects are definable subobjects $Z$ of 
$S \times \mathbb{A}_{k}^m$, for some $m$, where we denote by abuse of notation by $\mathbb{A}_{k}^m$
the object whose $K$-points are $K^m$, for every field extension $K$ of $k$.
For such a $Z$
denote by $p_Z : Z \to S$ the projection to $S$.

We define
$\RDef_{\acf,S}^{\Gamma}$ as the category whose objects
are objects $Z$ of $\RDef_{\acf,S}$ endowed with a definable $\Gamma$-action $\Gamma \times Z \to Z$ 
leaving the fibers of $p_Z$ stable and satisfying the following condition:
for any point $x$ of $S$, there exists a finite partition of the fiber $Z_x$ into Zariski locally closed subsets
$W_i$ which are stable under the $\Gamma$-action and the 
$\Gamma$-action $\Gamma \times W_i \to W_i$ is algebraic (viewing $\Gamma$ as a finite group scheme).
Morphisms in $\RDef_{\acf,S}^{\Gamma}$ are $\Gamma$-equivariant morphisms in $\RDef_{\acf,S}$.
We have a morphism
$\mathcal{P}^0 (S) \to K_0 (\RDef_{\acf,S}^{\Gamma})$
obtained
by composing
the morphism
$\mathcal{P}^0 (S) \to K_0 (\RDef_{\acf,S})$
with the morphism
$K_0 (\RDef_{\acf,S}) \to K_0 (\RDef_{\acf,S}^{\Gamma})$ 
which is induced by the functor sending a object in $\RDef_{\acf,S}$ to the same object endowed with the trivial $\Gamma$-action.
We define the ring $\cC_\acf^{\Gamma} (S)$ as the tensor product
$\cC_\acf^{\Gamma} (S) := K_0 (\RDef_{\acf,S}^{\Gamma}) \otimes_{\mathcal{P}^0 (S)} \mathcal{P} (S)$.

The integration theory developed in \cite{CL-2008} for $\cC_\acf (S)$ can be reproduced verbatim to $\cC_\acf^{\Gamma} (S)$.
In particular we have the following equivariant version of Proposition \ref{motfubini}:

\begin{proposition}\label{gfubini}The statement of Proposition \ref{motfubini} holds
with $\cC_\acf$ replaced by $\cC_\acf^{\Gamma}$.
\end{proposition}

We have a natural functor
$q : \RDef_{\acf,S}^{\Gamma} \to \RDef_{\acf,S}$
which sends an object $X$ in $\RDef_{\acf,S}^{\Gamma}$ to the quotient $X / \Gamma$.
It induces a map $q: K_0 (\RDef_{\acf,S}^{\Gamma}) \to K_0 (\RDef_{\acf,S})$ which can be extended by linearity to a map
$q: \cC_\acf^{\Gamma} (S) \to \cC_\acf (S)$.

\begin{proposition}\label{qcomi}
Let $X$ be a smooth algebraic variety of pure dimension over $k \llp t \rrp$
and let $\omega$ be a top degree form.
Let $S \subset \underline{X}$ be in $\GDef_{\acf,k}$.
Let  $\psi \in \cC^\Gamma_\acf(S)$ be integrable. We have
\[ q \Big(\int_{S} \psi \vert \omega \vert\Big)=  \int_S q (\psi) \vert \omega \vert. \]
\end{proposition}

\begin{proof}Commutation of $q$ with integration is clear at each  step in the construction of the motivic integral.
\end{proof}

Now let $S$ be a quasi-projective variety over $k$.
We denote by $\Var_S^{\Gamma}$ the category of
morphisms
$p: X \to S$ with $X$ quasi-projective endowed with an algebraic  $\Gamma$-action such that, for every point $x$
of $X$, the $\Gamma$-orbit of $x$ is contained in an affine subset of $p^{-1} (p (x))$.
Morphisms are $S$-morphisms equivariant with respect to the $\Gamma$-action.
We denote by $K_0 (\Var_S^{\Gamma})$ the corresponding  Grothendieck ring of varieties over $S$
and by $\cM_S^{\Gamma}$ its localization by the  class $\mathbb{L}$ of the affine line over $S$ endowed with the trivial $\Gamma$-action.

If $q: S \to T$ is a morphism of quasi-projective $k$-schemes,
composition with $q$ and pullback induce respectively morphisms 
$$q_! : \cM_S^{\Gamma} \to \cM_T^{\Gamma}$$
and
$$q^* : \cM_T^{\Gamma} \to \cM_S^{\Gamma}.$$

We assume from now on that $\Lambda$ is the field $\mathbb{Q} (\mu_r)$ with $r$ the order of the group $\Gamma$
and we denote by $\widehat \Gamma$ the group of characters
$\Gamma \to \Lambda^{\times}$.

Let $S$ be a quasi-projective $k$-scheme endowed with the trivial $\Gamma$-action.
Consider  $M \in \DAC (S, \Lambda)$ endowed with a $\Gamma$-action, that is, a group
morphism $a : \Gamma \to \Aut_{\DAC (S, \Lambda)} (M)$.
Fix $\delta \in \widehat \Gamma$
and set
$p_{\delta} := \vert \Gamma \vert^{-1} \sum_{\gamma \in \Gamma} \delta^{-1} (\gamma) a (\gamma)$.
The morphism $p_{\delta}$ is a projector. Since $\Lambda$ is a $\mathbb{Q}$-vector space,
the category $\DAC (S, \Lambda)$ is pseudo-abelian, cf.  \cite[Proposition 9.2]{Ayoub_ENS}.
Thus we can consider the image of  $p_{\delta}$ as an object in 
$\DAC (S, \Lambda)$, which we denote by $M^{\delta}$ and call
the $\delta$-isotypical component of $M$.

Let $p : X \to S$ be a quasi-projective $S$-scheme endowed with an algebraic  $\Gamma$-action
such that, for every point $x$
of $X$, the $\Gamma$-orbit of $x$ is contained in an affine subset of $p^{-1} (p (x))$.
By functoriality the object $p_! (\mathds{1}_X)$ 
is endowed with a $\Gamma$-action.
Thus we can consider its $\delta$-isotypical component $p_! (\mathds{1}_X)^{\delta}$.

By the same argument as the one in  \cite[Lemma 2.1]{ivo_seb}, there is 
a unique group morphism 
$$
\chi_{S, c}^{\delta} :
\cM_S^{\Gamma} \longrightarrow K_0(\DAC (S, \Lambda))
 $$
such that for $p : X \to S$ as above, $\chi_{S, c}^{\delta} ([X]) := [p_! (\mathds{1}_X)^{\delta}]$.
Furthermore
the diagrams 
 \begin{equation}\label{eq5}\xymatrix{  \cM_S^{\Gamma} \ar[d]^{q_!} \ar[r]^-{\chi_{S, c}^{\delta}} &  K_0( \DAC (S, \Lambda))  \ar[d]^{q_!}  \\
							 \cM_T^{\Gamma} \ar[r]^-{\chi_{T, c}^{\delta}} & K_0( \DAC (T, \Lambda)) 
}\end{equation}
and
\begin{equation}\label{eq6} \xymatrix{  \cM_T^{\Gamma} \ar[d]^{q^*} \ar[r]^-{\chi_{T, c}^{\delta}} &  K_0( \DAC (T, \Lambda))  \ar[d]^{q^*}  \\
							 \cM_S^{\Gamma} \ar[r]^-{\chi_{S, c}^{\delta}} & K_0( \DAC (S, \Lambda)) 
}\end{equation}
are commutative.

For  $S$ in $\GDef_{\acf,k}$, we have a natural morphism
$$
\vartheta^{\delta}_S:  \cC_\acf^{\Gamma} (S) \longrightarrow \prod_{x \in \vert S \vert} K_0( \DAC (k(x), \Lambda)) \otimes_{\mathbb{Z} [\mathbb{L}, \mathbb{L}^{-1}]}  \mathbb{A}
$$
sending
 $\varphi$ to $(\chi_{k(x), c}^{\delta} (\varphi (x)))$.
 We denote by
 $\cC_\mot^{\Gamma} (S)$
 the subring of
$ K_0( \DAC (k(x), \Lambda)) \otimes_{\mathbb{Z} [\mathbb{L}, \mathbb{L}^{-1}]}  \mathbb{A}$
generated by the images
$ \vartheta^{\delta}_S (\cC_\acf (S))$, for $\delta$ running over $\widehat \Gamma$.
Note that $\vartheta_S (\varphi) = \sum_{\delta \in \widehat \Gamma} \vartheta_S^{\delta} (\varphi)$
in  $\cC_\mot^{\Gamma} (S)$.

If $X$ is a quasi-projective  variety over $k$, then $\cC_\acf^{\Gamma} (\underline{X}_{{}})$ can be canonically identified with
$\cM_X^{\Gamma} \otimes_{\mathbb{Z} [\mathbb{L}, \mathbb{L}^{-1}]}    \mathbb{A}$.
For any character $\delta$, we have a morphism
$$
\chi_{X, c}^{\delta} \otimes \mathbb{A} : \cM_X^{\Gamma} \otimes_{\mathbb{Z} [\mathbb{L}, \mathbb{L}^{-1}]}    \mathbb{A}
\longrightarrow K_0( \DAC (X, \Lambda)) \otimes_{\mathbb{Z} [\mathbb{L}, \mathbb{L}^{-1}]}    \mathbb{A}.
$$
Again, it follows from Lemma \ref{evmot}
 that $\cC_\mot^{\Gamma} (\underline{X}_{{}})$ can be canonically identified with
the subring generated by the images of $\chi_{X, c}^{\delta} \otimes \mathbb{A}$ and that under that identification
$\vartheta^{\delta}_{\underline{X}_{{}}}$ and $\chi_{X, c}^{\delta} \otimes \mathbb{A}$
define the same morphisms
$\cC_\acf^{\Gamma} (\underline{X}_{{}}) \to \cC_\mot^{\Gamma} (\underline{X}_{{}})$.

Similarly as in the non-equivariant case, given a  morphism $q: X \to Y$ between quasi-projective  varieties over $k$,
one constructs morphisms
$q_! : \cC_\mot^{\Gamma} (\underline{X}_{{}}) \to \cC_\mot^{\Gamma} (\underline{Y}_{{}})$
and
$q^* : \cC_\mot^{\Gamma} (\underline{Y}_{{}}) \to \cC_\mot^{\Gamma} (\underline{X}_{{}})$
such that
the diagrams 
 \begin{equation}\label{eq7}\xymatrix{  \cC_\acf^{\Gamma} (\underline{X}_{{}}) \ar[d]^{q_!} \ar[r]^-{\vartheta_{\underline{X}_{{}}}} &  \cC_\mot^{\Gamma} (\underline{X}_{{}})  \ar[d]^{q_!}  \\
							 \cC_\acf^{\Gamma} (\underline{Y}_{{}}) \ar[r]^-{\vartheta_{\underline{Y}_{{}}}} & \cC_\mot^{\Gamma} (\underline{Y}_{{}}) 
}\end{equation}
and
\begin{equation}\label{eq8} \xymatrix{  \cC_\acf^{\Gamma} (\underline{Y}_{{}}) \ar[d]^{q^*} \ar[r]^-{\vartheta_{\underline{Y}_{{}}}} &  \cC_\mot^{\Gamma} (\underline{Y}_{{}})  \ar[d]^{q^*}  \\
							 \cC_\acf^{\Gamma} (\underline{X}_{{}}) \ar[r]^-{\vartheta_{\underline{X}_{{}}}} & \cC_\mot^{\Gamma} (\underline{X}_{{}})
}\end{equation}
commute.

Let $S$ be in $\Def_{\acf,k}$
and consider an integrable function $\varphi$ in $\cC_\acf^{\Gamma} (S)$. 
One deduces similarly as in Proposition \ref{indrel} that
$\vartheta_{\ast_{k}}^{\delta} \Bigl(\int_S \varphi \, \vert \omega_0\vert_S\Bigr) $
depends only on 
$\vartheta_S^{\delta} (\varphi)$
and we denote it by
$\int_S^{\mathrm{mot}, \delta} \varphi \, \vert \omega_0\vert_S$.
This extends to the global case as above:
if $X$  is a
smooth algebraic variety over $k\llp t \rrp$ of pure dimension, $\omega_X$ a top degree form on $X$ and
$\varphi$ a function in $\cC_\acf^{\Gamma} (\underline{X})$ such that $\varphi \, \vert \omega_X \vert$ is integrable, then
$\vartheta_{\ast_{k}}^{\delta} \Bigl(\int_{\underline{X}} \varphi \, \vert \omega_X\vert_S\Bigr)$
depends only on $\vartheta^{\delta}_{\underline{X}}  (\varphi)$
and we denote it by 
$\int_{\underline{X}}^{\mathrm{mot}, \delta}   \varphi \, \vert \omega_X \vert$.
We will also set %\todo{Can we define the lhs more directly?} FL je propose de laisser comme ça
$$
\int_{\underline{X}}^{\mathrm{mot}}   \varphi \, \vert \omega_X \vert =
\sum_{\delta \in \widehat{\Gamma}} \int_{\underline{X}}^{\mathrm{mot}, \delta}   \varphi \, \vert \omega_X \vert.
$$
If $Y$  is an algebraic variety over $k$ and $\varphi$ is a function in $\cC_\acf^{\Gamma} (\underline{Y}_{{}})$, 
one defines similarly
$\int_{\underline{Y}_{{}}}^{\mathrm{mot}, \delta}   \varphi$
and
$\int_{\underline{Y}_{{}}}^{\mathrm{mot}}   \varphi$.

We may now state the following final avatar of Proposition \ref{fubini} which we will use in this paper. It follows directly from Proposition \ref{gfubini} 
and the above discussion.

\begin{proposition}\label{eqmotfubini}\begin{enumerate}
\item
Let $X$ and $Y$ be 
smooth algebraic varieties over $k\llp t \rrp$ of pure dimension. 
Let $f : X \to Y$ be a smooth morphism. %We denote  by $\underline{f} : \underline{X} \to \underline{Y}$ the morphism induced by $f$.
Let $\omega_X$ and $\omega_Y$ be top degree forms on $X$ and $Y$, respectively, everywhere non-zero.
Let $\varphi \in \cC_\acf (\underline{X})$ such that $\varphi \, \vert \omega_X \vert$ is integrable.
 Let $\delta$ be in $\widehat \Gamma$.
There exists a definable subset $Z$ of $\underline{Y}$ of $K$-dimension at most $\dim Y - 1$
and
$\psi \in \cC_\acf^{\Gamma} (\underline{Y})$ with $\psi \, \vert \omega_Y \vert$ integrable such that,
for every point $y$  of $\underline{Y} \setminus Z$,
$$
\vartheta^{\delta}_{\underline{Y}} (\psi) (y) = \int_{\underline{X}_y}^{\mathrm{mot}, \delta} \varphi_{\vert \underline{X}_y} \, \vert \omega_X / f^* (\omega_Y) \vert_{\vert \underline{X}_y},
$$
and
$$
\int_{\underline{X}}^{\mathrm{mot}, \delta}  \varphi \, \vert \omega_X \vert =
\int_{\underline{Y}}^{\mathrm{mot}, \delta} \psi \, \vert \omega_Y \vert.
$$
\item
Let $X$  be a
smooth algebraic variety over $k\llp t \rrp$ of pure dimension. Let $Y$ be an algebraic variety over $k$.
Let  $\underline{f} : \underline{X} \to \underline{Y}_{{}}$ be a morphism in $\GDef_{\acf,k}$.
Let $\omega_X$ be a top degree form on $X$ which is everywhere non-zero.
Let $\varphi \in \cC_\acf (\underline{X})$ such that $\varphi \, \vert \omega_X \vert$ is integrable.
Then, there exists 
$\psi \in \cC_\acf^{\Gamma} (\underline{Y}_{{}})$ such that,
for every point $y$  of $\underline{Y}_{{}}$,
$$
\vartheta_{\underline{Y}_{{}}}^\delta (\psi) (y) = \int_{\underline{X}_y}^{\mathrm{mot}, \delta} \varphi_{\vert \underline{X}_y} \, \vert \omega_X  \vert_{\vert \underline{X}_y},
$$
and
$$
\int_{\underline{X}}^{\mathrm{mot}, \delta} \varphi \, \vert \omega_X \vert =
\int_{\underline{Y}_{{}}}^{\mathrm{mot}, \delta} \psi.
$$
\end{enumerate}
\end{proposition}

\subsection{Some useful facts}
We collect here some well-known facts that will be used later in the paper.

Let $X$ be quasi-projective variety over a field of characteristic zero on which a finite group 
$\Gamma$ acts algebraically and assume that, for every point $x$
of $X$, the $\Gamma$-orbit of $x$ is contained in an affine subset of $X$.
Denote by $h : X \to Y := X / \Gamma$ the quotient map.
The morphism $\mathds{1}_Y \to h_*h^* (\mathds{1}_Y) = h_* (\mathds{1}_X)$ given by the unit of the adjunction
induces an isomorphism
\begin{equation}\label{motinv} \mathds{1}_Y \cong h_* (\mathds{1}_X)^1,\end{equation}
the isotypical part of $h_* (\mathds{1}_X)$ for the trivial character. This can be checked directly from the definitions, or by proper base change one can reduce
to the case where $Y$ is a point which is even simpler.
Note that since $h$ is proper one can identify
$h_* (\mathds{1}_X)$ with $h_! (\mathds{1}_X)$.

If $\pi : X \to \Spec k$ is an algebraic variety over $k$, we shall consider its homological Chow motive  $M (X) := \pi_! \pi^!  (\mathds{1}_{\Spec k})$.

\begin{lemma}\label{isoge}
Let $k$ be a field of characteristic zero.
Let $A$ and $B$ be connected commutative algebraic groups over $k$ and let 
$p : A \to B$ be an isogeny. 
Then the 
morphism
$M (A) \to M (B)$ induced by $p$ is an isomorphism.
\end{lemma}

\begin{proof}It is enough to check that for any non-zero natural number $n$, the morphism
$[n] : M (A) \to M (A)$  induced by multiplication by $n$ is an isomorphism, which  follows from 
Theorem 7.1.1 (1) in \cite{AEH} in the semiabelian case
and from Theorem 3.3 (4) in 
\cite{AHP} in general.
\end{proof}

\begin{lemma}\label{isomot}Let $k$ be a field of characteristic zero.
Let $p : A \to B$ be an isogeny between connected commutative algebraic groups over $k$ with kernel $\Gamma$.
Denote by  $\pi_A$ and $\pi_B$ the structural morphisms.
The finite abelian group $\Gamma$
acts on the object
$p_! (\mathds{1}_A)$ of
$\DAC (B, \Lambda)$ and for
any character $\delta : \Gamma \to \Lambda^{\times}$ we denote by
$p_! (\mathds{1}_A)^{\delta}$ its $\delta$-isotypical component.
Then $\pi_{B!} (p_! (\mathds{1}_A)^{\delta}) = 0$ if
$\delta$ is non trivial and
$\pi_{B!} (p_! (\mathds{1}_A)^{\delta}) = \pi_{B!} (p_! (\mathds{1}_A))$ if
$\delta$ is trivial.
\end{lemma}

\begin{proof}As we already observed,
the morphism 
\begin{equation}\label{eq11}
\mathds{1}_B \longrightarrow  p_*(\mathds{1}_A)
\end{equation}
induces an isomorphism between  $\mathds{1}_B$ and $p_* (\mathds{1}_A)^1$.
Dualizing this morphism, one gets a morphism
\begin{equation}\label{eq12}
p_! \pi_A^! (\mathds{1}_{\Spec k}) \longrightarrow   \pi_B^!  (\mathds{1}_{\Spec k})
\end{equation}
whose restriction to  
$(p_! \pi_A^! (\mathds{1}_{\Spec k}))^{\delta}$ is zero
when $\delta$ is non trivial and an isomorphism when $\delta$ is trivial.
Applying $\pi_{B!}$ to both sides
of (\ref{eq12}) one gets a morphism $\pi_{A!} \pi_A^! (\mathds{1}_{\Spec k}) \to \pi_{B!} \pi_B^! (\mathds{1}_{\Spec k})$ which 
is an isomorphism by 
Lemma \ref{isoge}.
It follows that
$\pi_{B!} ((p_! \pi_A^! (\mathds{1}_{\Spec k}))^{\delta})$ is zero
when $\delta$ is non trivial and isomorphic to 
$\pi_{B!} \pi_B^! (\mathds{1}_{\Spec k})$ when $\delta$ is trivial.
The statement follows since
$\pi_A^! (\mathds{1}_{\Spec k}) \simeq \mathds{1}_A (d) [2d]$
and
$\pi_B^! (\mathds{1}_{\Spec k}) \simeq \mathds{1}_B (d) [2d]$, with $d$ the dimension of both $A$ and $B$.
\end{proof}

\begin{corollary}\label{autab}
Let $k$ be a field of characteristic zero.
Let $\pi : A \to \Spec k$ be a connected commutative algebraic group
and let  $\xi \in A (k)$ be a torsion point. Denote by 
$g_\xi$ the endomorphism given by translation by  $\xi$.
Then $g_\xi$ acts trivially on
$\pi_! (\mathds{1}_{A}) \simeq M (A) (-d) [2d]$.
\end{corollary}

\begin{proof}Follows directy from Lemma \ref{isomot} applied to
the isogeny $A \to B$ with $B$ the quotient of $A$ by the subgroup generated by $\xi$.
\end{proof}

%We finish the section with some computational results, that will help manipulate the motivic integrals in the next section. The first is a motivic analogue of \cite[Theorem 4.8]{Yasuda:2014aa}

%\begin{lemma}\label{quotcomp} Let $f:X \rightarrow Y$ be a morphism of smooth $k\llp t \rrp$-varieties. Assume that a finite group $\Gamma$ acts on $X$ generically free, such that $f$ is $\Gamma$-invariant and the induced morphism $X/\Gamma \rightarrow Y$ is birational. Let $\omega$ be a top-dimensional form on $Y$ and $\phi \in \cC_\acf(Y)$ such that $\phi|\omega|$ is integrable. Then we have
%\[ \int_{X} f^*\phi |f^*\omega| = |\Gamma| \int_Y \phi|\omega| \in \cC_\acf(k).   \]
%The same holds in $\cC_\acf^\Gamma$.\todo{Proof}
%\end{lemma}
\subsection{Functions from quasi-finite morphisms}\label{quasifinite}

Consider a quasi-finite morphism 
\[ f:X \longrightarrow Y   \] between $k\llp t \rrp$-varieties  of pure dimension.
Then any  $S \subset \underline{X}$ in $\GDef_{\acf,k}$ is isomorphic to an
object in 
 $\RDef_{\acf,\underline{Y}}$ by an isomorphism commuting with the map to $\underline{Y}$. Indeed, this follows from Corollary 7.2.3 of  \cite{CL-2008} (there is an unfortunate typo, $r$ should read $0$).
 
Thus, we may attach to  $S$ a constructible motivic function in $\cC_\acf(\underline{Y})$ (or in $\cC_\acf(\underline{f}(S))$
which we still denote by $S$. 
Assume now that  $\Gamma$ acts on $X$,  $f$ is $\Gamma$-invariant and $S$ is stable under $\Gamma$.
We define similarly a function in $\cC_\acf^\Gamma(\underline{Y})$ (or in $\cC_\acf^\Gamma(\underline{f}(S))$ which we will also denote by $S$.

\begin{rmk}\label{support} The support $\underline{f}(S)$ of $S$ as a function is in general smaller than the subassignment associated with the image $\underline{f(S)}$. In fact a point $x \in \underline{f(S)}(K)$ belongs to $\underline{f}(S)$ if the fiber $f^{-1}(x)$ admits a $K'\llp t \rrp$-rational point for some extension $K' /K$.
\end{rmk}

The following lemma allows us to compute with these functions.

\begin{lemma}\label{qfint} Let $ f:X \longrightarrow Y$ be an \'etale morphism between smooth $k\llb t \rrb$-varieties, both of pure relative dimension $d$, and 
let $\omega$ be a degree $d$ differential form on $Y$ which generates the $k\llb t \rrb$-module of degree $d$ differential forms everywhere. Then we have
\[ \int_{\underline{Y}_{\circ}} X |\omega| = \BL^{-d} [X_k],  \]
with $X_k$ the special fiber of $X$.
If a finite abelian group $\Gamma$ acts on $X$ and $f$ is $\Gamma$-invariant, the same holds $\Gamma$-equivariantly.
\end{lemma}

\begin{proof}
Consider the definable morphism $\underline{\ev_X}: \underline{X}_{\circ} \to \underline{X_k}_{{}}$
 given by reducing modulo $t$, and define similarly $\underline{\ev_Y}$.
 The morphism $(\underline{f}, \underline{\ev_X}): \underline{X}_{\circ} \to \underline{Y}_{\circ} \times \underline{X_k}_{{}}$
 is an isomorphism since $f$ is \'etale.
 Thus the function in $\cC_\acf(\underline{Y}_{\circ})$ associated to $\underline{X}_{\circ}$ is of the form
 $\underline{\ev_Y}^* (\varphi)$ with $\varphi$ the function $\cC_\acf(\underline{Y_k}_{{}})$ associated to $\underline{X_k}_{{}}$.
 The lemma follows now directly from Lemma \ref{reduc}.
 \end{proof}

%\begin{proof}Since $f$ is \'etale
%$f^* \omega$ is a generator of $\Omega^d_{Y \vert k\llb t \rrb}$ at each point.
%Th
%us, by Lemma \ref{reduc}, 
%we have 
%int_{\underline{Y}} X |\omega|

%Proposition \reffubini}

%\end{proof}

\subsection{Twisting}\label{twsec}

Let $\Gamma$ and $\Gamma'$ be two finite commutative groups.
Let $S$ be an object of $\GDef_{\acf,k}$ and let $\varphi \in \RDef_{\acf,S}^{\Gamma \times \Gamma'}$, $\psi \in \RDef_{\acf,S}^{\Gamma}$. 
We let $\Gamma$ act anti-diagonally on the product
$\psi \times \varphi$ in $\RDef_{\acf,S}$ by $\gamma (a, b) = (\gamma a, \gamma^{-1} b)$ and consider the quotient 
\[\psi \times^\Gamma \varphi :=  \psi \times_S \varphi / \Gamma\] 
as an object of $\RDef_{\acf,S}^{\Gamma'}$.
We get this way a morphism
$\cdot \times^\Gamma \varphi: \RDef_{\acf,S}^{\Gamma} \to \RDef_{\acf,S}^{\Gamma'}$
which extends by $\mathcal{P} (S)$-linearity to a morphism $\cdot \times^\Gamma \varphi: \cC_\acf^{\Gamma} (S) \to \cC_\acf^{\Gamma'} (S)$.
%This construction extends $\mathcal{P} (\ast_{k})$-linearly to a morphism $\tau_{\varphi}: \cC_\acf^{\Gamma} (S) \to \cC_\acf^{\Gamma'} (S)$ for any $\varphi \in  \cC_\acf ^{\Gamma \times \Gamma'}(\ast_{k})$.
We record two lemmas for later use. The first explains the effect of unramified twisting on integrals.

\begin{lemma}\label{unrtw}Let $X$ be a smooth algebraic variety of pure dimension over $k \llp t \rrp$
and let $\omega$ be a top degree form.
Let $S \subset \underline{X}$ be in $\GDef_{\acf,k}$.
Let $\varphi \in \cC^{\Gamma \times \Gamma'}_\acf(\ast_{k})$ and $\psi \in \cC^\Gamma_\acf(S)$ be integrable. We have
\[ \int_{S} \psi \times^\Gamma p^*(\varphi) \vert \omega \vert = \left( \int_S \psi \vert \omega \vert \right) \times^\Gamma \varphi,\]
where $p:S \to \ast_{k}$ is the morphism to the point.
\end{lemma}

\begin{proof}Follows directly from the projection formula Proposition 13.2.1 (2) in \cite{CL-2008} and
Proposition \ref{qcomi}.
\end{proof}

%We will use the notation $\psi \times^\Gamma \varphi$ more generally for any $\varphi \in \RDef_{\acf,\ast_{k}}^{\Gamma \times \Gamma'}$ and $\psi \in \RDef_{\acf,S}^{\Gamma}$. 

The second lemma is about the effect of twisting on isotypical components. For this we assume $S$ to be a finite type $k\llp t\rrp$-scheme and $T$ a $\Gamma$-torsor on $S$. We can consider $\underline{T}$ as an element in $\cC^{\Gamma \times \Gamma}_\acf(\underline{S})$ by doubling the $\Gamma$-action.

\begin{lemma}\label{twnotw} Let $\delta:\Gamma \rightarrow \mu_n$ be a character such that the push-forward $\delta_*T \in H^1(S,\mu_n)$ is trivial. Then we have for any $\psi \in \cC^\Gamma_\acf(\underline{S})$ the equality
\[ \vartheta^{\delta}_{\underline{S}}\left( \psi \times^\Gamma \underline{T} \right) = \vartheta^{\delta}_{\underline{S}}\left( \psi \right). \]
\end{lemma}
\begin{proof}
We use the observation, that for any $\psi \in \cC^\Gamma_\acf(\underline{S})$ we have 
\[\vartheta^{\delta}_{\underline{S}}\left( \psi \right) = \vartheta^{\id}_{\underline{S}}\left( \psi \times^\Gamma \underline{\mu_n} \right),\]
where on the right hand side $\Gamma$ acts on $\mu_n$ through $\delta$ and $\id:\mu_n \to \mu_n$ is the identity. Indeed, by definition of the isotypical component and \eqref{motinv} we see that both sides are equal to the $\widetilde{\delta}$-component of $\psi \times \mu_n$ with the product action of $\Gamma \times \mu_n$, where $\widetilde{\delta}$ is the character of $\Gamma \times \mu_n$ given by $\widetilde{\delta}(\gamma,\xi) = \delta(\gamma)\xi$.

Now by assumption we have an isomorphism of $\mu_n$-torsor $T \times^\Gamma \mu_n \cong \mu_n$ and therefore 
\[\vartheta^{\delta}_{\underline{S}}\left( \psi \times^\Gamma \underline{T} \right) = \vartheta^{\id}_{\underline{S}}\left( \psi \times^\Gamma \underline{T}\times^\Gamma \underline{\mu_n} \right) = \vartheta^{\id}_{\underline{S}}\left( \psi \times^\Gamma \underline{\mu_n} \right)=\vartheta^{\delta}_{\underline{S}}\left( \psi \right). 
\qedhere \]
\end{proof}

\section{Orbifold volumes}\label{orbivol}\label{sec3}

\subsection{Orbifold measure and stringy invariants} 
We start with a well known lemma.

\begin{lemma}\label{reduc} Consider $X\to \Spec k  \llb t \rrb$ a smooth separated $k \llb t \rrb$-scheme of finite type, purely of relative dimension $d$.
Let $\omega$ be degree $d$ differential form on $X$ which generates the $k\llb t \rrb$-module of degree $d$ differential forms everywhere.
%Denote by 
 %$\ev: X \to X_k$ the specialization map to the  special fiber $X_k$.
 Let $X_k$ denote the special fiber of $X$ and
 consider
 the definable morphism
 $\underline{\ev}: \underline{X}_{\circ} \to \underline{X_k}_{{}}$
 given by reducing modulo $t$.
Then, for every  $\varphi \in \cC_\acf (\underline{X_k}_{{}})$,
we have
\[
\int_{\underline{X}_{\circ}} \underline{\ev}^* (\varphi) \, \vert \omega \vert
=
\mathbb{L}^{-d} \int_{\underline{X_k}_{{}}} \varphi 
\] 
in $\cC_\acf (\ast_{k})$.
%Here by abuse of notation we still denote by $\underline{\mathcal{X}}$ the corresponding definable subassignment.
The same statement holds with $\cC_\acf$ replaced by 
$\cC_\mot$, $\cC_\acf^{\Gamma}$ or $\cC_\mot^{\Gamma}$.
\end{lemma}

\begin{proof}
%The induced morphism
%$\underline{\ev}: \underline{X} \to \underline{X_k}$ is definable, 
By the projection formula Proposition 13.2.1 (2) in \cite{CL-2008}, we have the equality
\[
\int_{\underline{X}_{\circ}} \underline{\ev}^* (\varphi) \, \vert \omega \vert
=
\int_{\underline{X_k}_{{}}} \varphi \cdot \psi
\]
with  $\psi$ the function
$x \mapsto \int_{\underline{\ev}^{-1}(x)} \vert \omega \vert$.
Since $X$ is smooth purely of relative dimension $d$ and $\omega$ is a 
$k \llb t \rrb$-generator everywhere, $\psi$ is constant equal to $\mathbb{L}^{-d}$.
\end{proof}

We will now extend Lemma \ref{reduc} to finite quotient singularities. For this let $M$ be a smooth variety over $k\llb t \rrb$  purely of relative dimension $d$ and $\Gamma$ a finite abelian group acting generically freely on $M$. 
We assume  that  the $\Gamma$-orbit of every point is contained  in an affine open subset, which is the case when $M$ is quasi-projective.
Let $X=M/\Gamma$ denote the geometric quotient and by $U \subset X$ the maximal open subvariety where the quotient morphism $\pi:M \to X$ is a $\Gamma$-torsor. We will always assume that $k$ contains a primitive root of unity $\xi$ of order $\vert \Gamma \vert$.

By abuse of notation we write $IX$ for the coarse moduli space of the inertia stack of $[M/\Gamma]$. More explicitly $IX$ can be identified with the disjoint union
\[  IX = \bigsqcup_{\gamma\in\Gamma} M^\gamma/\Gamma, \]
where $M^\gamma \subset M$ denotes the fixed point locus of $\gamma$. We will only be interested in the special fiber $IX_k
= \sqcup_{\gamma\in\Gamma} M^\gamma_k/\Gamma$ of $IX$. 

We define the definable subassignements $X^\natural, M^\natural \in \GDef_{\acf,k}$ by
\[ X^\natural (\overline{K}) = X(\overline{K}\llb t \rrb) \cap U(\overline{K}\llp t \rrp) \ \ \ \text{ and } \ \ \ M^\natural (\overline{K})= M(\overline{K}\llb t \rrb) \cap \pi^{-1}U(\overline{K}\llp t \rrp)  \]
for any algebraically closed field $\overline{K}$ containing $k$.

We will always assume that $M$ admits a nowhere vanishing global $d$-form $\omega$ which is invariant under the action of $\Gamma$ (in particular $\Gamma$ preserves the canonical bundle of $M$). Then $\omega$ descends to a $d$-form on $\omega_{orb}$ on $U$, which we then can integrate over $X^\natural$. To describe the resulting volume we first construct a definable specialization morphism 
\[e: X^\natural \rightarrow \underline{IX_k}_{{}}. \]

\begin{construction}[Construction of $e$]\label{econst} %Let $K/k$ be a field extension and $x \in X^\natural (K)$. 
Let $x \in \vert X^\natural \vert$ and  set $K = k (x)$.
Thus $x$ gives rise to a morphism
 \[x: \Spec(K\llb t \rrb) \rightarrow  X  \] 
which maps the generic point to $U$. The fiber $T$ of $\pi:M \to X$ over $x_{|\Spec(K\llp t \rrp)}$ is by construction a $\Gamma$-torsor. In particular there exists a finite Galois extension $L/K\llp t \rrp$  and a positive integer $r$ such that 
%\[T = \Spec \left(\prod L \right).\]
\[T = \Spec \left(L^{\times r} \right)\]
with $L^{\times r}$ the product of $r$ copies of $L$.
We write $T_{K\llb t \rrb}$ for the spectrum of the normalization of $K\llb t \rrb$ inside $L^{\times r}$, explicitly $T_{K\llb t \rrb} = \Spec \left( K\llb t \rrb^{\times r}\right)$. We obtain a commutative diagram

\[ \xymatrix{ T \ar[r] \ar[d] & T_{K\llb t \rrb} \ar[r]^{\tilde{x}} \ar[d] & M \ar[d]^\pi \\
\Spec(K\llp t \rrp) \ar[r] & \Spec(K\llb t \rrb) \ar[r]^{\ \  \ \  x} & X,}
\]
where we used properness of $\pi$ to obtain the $\Gamma$-equivariant morphism $\tilde{x}$.

The inertia group $I\subset \Gal(L/K\llp t \rrp) \subset \Gamma$ is naturally isomorphic to a subgroup of the group of roots of unity in $K$. Our choice of primitive root of unity $\xi$ thus gives a generator $\gamma$ of $I$. By definition of the inertia group, $I$ acts trivially on the special fiber of $ T_{K\llb t \rrb} $ and thus the image of $\tilde{x}_{|\Spec(K)}$ is 
a $K$-rational point of  $M^\gamma_k$. Passing to the quotient we obtain a well-defined $K$-point in $M_k^\gamma/\Gamma$ which we denote by $e(x)$.
\end{construction}

\begin{rmk}\label{compdir} The choice of $\xi$ gives for every extension $K/k$ a splitting 
\begin{equation}  H^1(K\llp t \rrp,\Gamma) \cong H^1(K,\Gamma) \oplus \Gamma.     \end{equation}
Torsors coming from $ H^1(K,\Gamma)$ are called unramified, as they extend to $\Spec(K\llb t\rrb)$.
The class of the $\Gamma$-torsor $T$ appearing in  Construction  \ref{econst} is given by $[T_K,\gamma]$ under this splitting. 

For any $\gamma \in \Gamma$ we will denote by $T_\gamma$ the  $\Gamma$-torsor defined over $k$
\begin{equation}\label{tgamma} T_\gamma = \Spec(k \llp t^{\frac{1}{r}} \rrp) \times_{\Spec(k\llp t \rrp)}  \Gamma / \langle \gamma \rangle,\end{equation}
where $r$ denotes the order of $\gamma$ and $\gamma$ acts on $k \llp t^{\frac{1}{r}} \rrp$ through $\xi^{\frac{|\Gamma|}{r}}$.
Its  class is given by $[0,\gamma]$.  %The component in $\Gamma/\Gamma$ to which a point  $x \in X^\circ$ gets sent under $e$ can also be computed directly as the class of the $\Gamma$-torsor $\pi^{-1}(x)$ in $H^1(k\llp t \rrp,\Gamma)$ under the isomorphism $\Gamma/\Gamma \cong H^1(k\llp t \rrp,\Gamma)$
\end{rmk}

%We are now in position to define the specialization morphism 
%$e: X^\natural \rightarrow \underline{IX_k}$. Let $x \in X^\natural (K)$ with
%$K$ an algebraically closed field containing $k$. We set
%$e (x) := [\tilde x ] \in \underline{IX_k} (K)$.

\begin{proposition}\label{edefin}The mapping
$e: \vert X^\natural\vert \rightarrow \vert {IX_k} \vert$ is definable, that is, it is induced by a  definable morphism
$e: X^\natural \rightarrow
\underline{IX_k}_{{}}$.
\end{proposition}

\begin{proof}It is enough  to prove that, for each $\gamma \in \Gamma$, 
the set 
$e^{-1} \vert {IX_k} \vert$ is definable, that is of the form 
$\vert X^{\natural, \gamma}\vert$ with $X^{\natural, \gamma}$ a definable subset of $X^{\natural}$.
Indeed, $X^\natural$ will be then the finite disjoint union of the sets 
$X^{\natural, \gamma}$ and the maps 
$X^{\natural, \gamma} \to \underline{M_k^\gamma/\Gamma}_{{}}$ are clearly definable.
Since a point $x$ in $X^\natural$ belongs to $e^{-1} \vert {IX_k} \vert$ if and only its fiber
$\pi^{-1} (x)$ is isomorphic, as a definable $\Gamma$-set, to $T_\gamma$ over some finite extension $L$ of the field $k(x)$,
the definability of 
$e^{-1} \vert {IX_k} \vert$ is a direct consequence of compactness
in first order logic. 
Indeed, consider a point $x$ in $e^{-1} \vert {IX_k} \vert$ and a finite extension $L$ of  $k(x)$ such that 
$\pi^{-1} (x)$ is isomorphic, as a definable $\Gamma$-set, to $T_\gamma$ over  $L$. Let $h$ be such an isomorphism.
Assume first $L = k (x)$. Working in affine charts, $h$ can be expressed by a finite family of  polynomials with coefficients in $k (x)$,
say $N$ polynomials in at most $r$ variables and degree at most  $d$, and the existence of such an isomorphism $h$ can be expressed by a formula 
$\varphi_{x, \gamma}$ in the language $\dplkx$.
It follows that there exists
a formula $\psi_{x, \gamma}$  in the language $\dplk$ which is satisfied
when one evaluates its free variables on the coordinates of $x$,
and such that if it is satisfied when evaluating the coordinates of some other point $y$, then 
there exists a definable isomorphism over $k(y)$ between $\pi^{-1} (y)$ and  $T_\gamma$.
In general, fix a positive integer $m$. One can encode field extensions $L$ of degree $m$ of a perfect field $E$
through the minimal polynomial of a generator, as explained in  Section 3.2 of 
\cite{chl}. In particular, the existence of a field extension of degree $m$ can be expressed as the existence
of an $m$-tuple $(a_1, \cdots, a_m)$ satisfying the definable condition expressing the irreducibility of the corresponding monic polynomial
$P = x^m + a_1 x^{m - 1} + \cdots a_m$
and working in the basis $(1, x, \cdots x^{m-1})$ of $L = E [x] / P$ allows to express definable conditions  involving polynomials with coefficients in $L$ as definable conditions on
tuples of elements of  $E$. In our case the 
existence of an isomorphism over $L$ may thus be rephrased as a definable condition on
tuples of elements of  $k (x)$. Thus, there exists 
a formula $\psi_{x, \gamma, m}$  in the language $\dplk$ which is satisfied
when one evaluates its free variables on the coordinates of $x$,
and such that furthermore, for any point $y$ in the given affine open,
if it is satisfied when evaluating the coordinates of $y$, then 
there exists an extension $L$ of degree $m$ of $k(y)$ and a definable isomorphism over $L$ between $\pi^{-1} (y)$ and  $T_\gamma$.
The set of all points $y$ satisfying $\psi_{x, \gamma, m}$ 
is a $\dplk$-definable subset $Z_{\psi_{x, \gamma, m}}$ 
of  $X^\natural$. By construction $X^\natural$ is covered by the set of all such definable sets $Z_{\psi_{x, \gamma, m}}$.
Thus, by compactness, $X^\natural$ is the union of finitely  many such definable sets, say $Z_{\psi_{x_i, \gamma_i, m_i}}$.
For each $\gamma$, define $X^{\natural, \gamma}$ as the union of the sets $Z_{\psi_{x_i, \gamma_i, m_i}}$ with $\gamma_i = \gamma$. It is a definable set 
and 
$\vert X^{\natural, \gamma}\vert = e^{-1} \vert {IX_k} \vert$.
\end{proof}

%\begin{rmk}[Twisting by torsors]\label{tortw} For any variety $Y$ over some base $B$ and any $\Gamma$ torsor $T/B$ we define the twist
%\[ Y \times_B^\Gamma T = Y \times_B T / \Gamma,   \]
%where we use the anti-diagonal $\Gamma$-action, $\gamma(y,b) = (\gamma y,\gamma^{-1} b)$. Twisting is compatible with the group structure on $H^1(B,\Gamma)$ i.e. for $\Gamma$-torsors $T_1,T_2$ we have
%\[ [T_1 \times^\Gamma T_2] = [T_1] + [T_2].\]
%\end{rmk}

\begin{definition}[Weight function]\label{weight} For any point $x \in M_k^\gamma$ consider the action of $\gamma$ on the tangent space $T_xM_k$. There are unique positive integers $1 \leq c_1,\dots,c_d \leq |\Gamma|$ such that the eigenvalues of $\gamma$ are given by $\xi^{c_1},\dots, \xi^{c_d}$. We define the weight $w_x(\gamma)$ as
\[w_x(\gamma)  = \frac{1}{|\Gamma|} \sum_{i=1}^d c_i.\]
Clearly $w_x(\gamma)$ is locally constant as a function in $x$. For simplicity we will further assume that $w_x(\gamma)$ is constant on the components of $M_k^\gamma$ and write $w(\gamma)$ for its value on $\M_k^\gamma$. Notice also that by assumption $\Gamma$ preserves the canonical bundle of $M$ and hence $w(\gamma) \in \BZ$.
\end{definition}

%\begin{theorem}\label{orbvol} For every $\varphi \in \cC_\acf (X_k)$ and $\gamma \in \Gamma$ we have
%\[ \int_{e^{-1}(M^\gamma/C(\gamma))} e^*\varphi |\omega_{orb}| =\int_{M_k^\gamma/C(\gamma)} \mathbb{L}^{-w} \varphi.  \]
%The same statement holds with $\cC_\acf$ replaced by $\cC_\mot$, $\cC_\acf^{\mu_n}$ or $\cC_\mot^{\mu_n}$.
%\end{theorem}

%  Now let $\omega$ be a top degree form on $Y$.
 % It follows directly from Proposition \ref{gfubini}  that
%  \begin{equation}\label{tf}
 % \int_{\underline{f}(S)}  S \,  \vert \omega \vert = 
 % \int_S \vert f^{\ast} \omega \vert.
%\end{equation}

%Indeed, for any point $y$ in $\underline{f}(S)$, the integral $\int_{S_y}  \vert f^* (\omega) / f^* (\omega) \vert_{\vert S_y}$ is equal to the class of 
%$S_y$ in $\cC_\acf^\Gamma(S_y))$.

In our case we are interested in integrating the twists $M^\natural \times^\Gamma T_{\gamma^{-1}}$ over $X^\natural$.
Here by $M^\natural \times^\Gamma T_{\gamma^{-1}}$
we mean of course
the image
of $M^\natural \times \underline{T_{\gamma^{-1}}}_{{}}$ in
$\underline{M \times^\Gamma T_{\gamma^{-1}}}$.

By construction of the morphism $e$ in \ref{econst} and Remark \ref{compdir} we see that $M^\natural \times^\Gamma T_{\gamma^{-1}}$ is supported on $X^{\natural, \gamma} :=e^{-1} (\underline{M_k^\gamma/\Gamma}_{{}}) \subset X^\natural$.

\begin{theorem}\label{equivorb} For every $\gamma \in \Gamma$, we have 
\[\int_{X^{\natural, \gamma}}^{\mathrm{mot}}  M^\natural \times^{\Gamma} T_{\gamma^{-1}} |\omega_{orb}| = \BL^{-w(\gamma)} [M_{k}^\gamma] \]
in $\cC_\mot^\Gamma(\ast_{k})$. %tors
In particular for any $\varrho \in \widehat{\Gamma}$,
\begin{equation}\label{isorb}\int_{X^{\natural, \gamma}}^{\mathrm{mot}, \varrho}  M^\natural \times^{\Gamma} T_{\gamma^{-1}} |\omega_{orb}| = \BL^{-w(\gamma)} \vartheta^{\varrho}_{\ast_{k}} ([M_{k}^\gamma]) \end{equation}
in $\cC_\mot(\ast_{k})$. %tors
\end{theorem}

%\begin{rmk} If $\varrho \in \widehat{\Gamma}$ is a character satisfying $\varrho(\gamma^{-1}) = 1$ for some $\gamma \in \Gamma$, one can simplify \eqref{isorb} as follows. Let $n$ be the order of $\Gamma$ and write $(\mu_n,\varrho)$ for the $\Gamma \times \mu_n$-variety $\mu_n$ where $\mu_n$ acts by multiplication and $\Gamma$ through $\varrho$. Finally write $id:\mu_n \to \mu_n$ for the character given be the identity. Then we have
%\begin{equation}\label{simpg} \int_{e^{-1}(M^\gamma/\Gamma)}^{\mathrm{mot}, \varrho} M^\natural \times^{\Gamma} T_{\gamma^{-1}} |\omega_{orb}|= \int_{e^{-1}(M^\gamma/\Gamma)}^{\mathrm{mot}, id} M^\natural \times^{\Gamma} (\mu_n,\varrho) |\omega_{orb}|.\end{equation}

%Indeed by definition of $\int^{\mot,\varrho}$ and Lemma \ref{unrtw} we have 
%\begin{align*}\int_{e^{-1}(M^\gamma/\Gamma)}^{\mathrm{mot}, \varrho} M^\natural \times^{\Gamma} T_{\gamma^{-1}} |\omega_{orb}| &= \left(\int_{e^{-1}(M^\gamma/\Gamma)}^{\mathrm{mot}, id} M^\natural \times^{\Gamma} T_{\gamma^{-1}} |\omega_{orb}|\right) \times^{\Gamma} (\mu_n,\varrho)\\
%&= \int_{e^{-1}(M^\gamma/\Gamma)}^{\mathrm{mot}, id} M^\natural \times^{\Gamma} T_{\gamma^{-1}}\times^{\Gamma} (\mu_n,\varrho) |\omega_{orb}|.  \end{align*}
%Now the assumption $\varrho(\gamma^{-1}) = 1$ implies 
%\[M^\natural \times^{\Gamma} T_{\gamma^{-1}}\times^{\Gamma} (\mu_n,\varrho) \cong M^\natural \times^{\Gamma} (\mu_n,\varrho).\]
%\end{rmk}

\subsection{Proof of Theorem \ref{equivorb}}

We start by proving a local version of Theorem \ref{equivorb}.

\begin{proposition}\label{linearaffine} Assume $\Gamma$ acts linearly on $M = \BA^d_{k\llb t \rrb}$ and fix $\gamma \in \Gamma$. Let $0$ denote the image of the origin in $\left( \BA_k^d\right)^\gamma / \Gamma \subset IM_k$. We have the relation
\[  \int^{\mathrm{mot}}_{e^{-1}(0)} M^\natural\times^\Gamma T_{\gamma^{-1}} |\omega_{orb}| =   \BL^{-w(\gamma)} \] 
in $\cC^\Gamma_\mot(\ast_{k})$, where $\BL$ is considered with the trivial $\Gamma$-action.
Moreover, 
for any $\delta \in \widehat{\Gamma}$ which is non-trivial,
\[ \int^{\mathrm{mot}, \delta}_{e^{-1}(0)} M^\natural\times^\Gamma T_{\gamma^{-1}} |\omega_{orb}| = 0.\] 
\end{proposition}
\begin{proof} Let $I \subset \Gamma$ be the cyclic group generated by $\gamma$ and $r=|I|$. We may assume that $\Gamma$ acts diagonally on $\BA^d_{k\llb t \rrb}$ 
 and that $\gamma$ has eigenvalues  $\xi^{c_1},\dots,\xi^{c_d}$ for unique integers $1 \leq c_i \leq d$. We then use \eqref{tgamma} to get the description
 \[\BA^d_{k \llp t \rrp} \times^{\Gamma} T_{\gamma^{-1}} \cong \Spec\left(k \llp t^{1/r} \rrp [x_1,\dots,x_n]^I\right),\]
where $\gamma$ acts on $k \llp t^{1/r} \rrp$ by $\gamma \cdot t^{1/r} = \xi t^{1/r}$. We then define a $k\llp t \rrp$-morphism 
 \[\lambda_\gamma: \BA^d_{k \llp t \rrp} \to \BA^d_{k \llp t \rrp} \times^{\Gamma} T_{\gamma^{-1}}\]
 on the level of coordinate rings by
\[ f(x_d,\dots,x_d) \mapsto f(t^\frac{c_1}{r}x_1,\dots,t^\frac{c_n}{r}x_n). \]
As in \cite[2.3]{DL2002} one sees that $\lambda_\gamma$ is indeed defined over $k\llp t \rrp$, and since $t$ is invertible it is in fact an isomorphism.  We thus obtain a cartesian square
\[ \xymatrix{\BA^d_{k \llp t \rrp} \ar[r]^{\lambda_\gamma \ \ \ \ \ } \ar[d] & \BA^d_{k \llp t \rrp} \times^{\Gamma} T_{\gamma^{-1}} \ar[d] \\
\BA^d_{k \llp t \rrp}/\Gamma \ar[r]^{\bar{\lambda}_\gamma } & \BA^d_{k \llp t \rrp}/ \Gamma.}\]

Now let $Y$ be the image of $\underline{M}$ in $\underline{M/\Gamma}$. It is a definable subassignment of $\underline{M/\Gamma}$.
Set $Y^\natural = Y \cap (\BA^d/\Gamma)^\natural$.  Then by \cite[2.7]{DL2002} we have a definable bijection 
\[\bar{\lambda}_\gamma: Y^\natural \rightarrow e^{-1}(0),    \]
and hence
\[   \int^{\mathrm{mot}}_{e^{-1}(0)} M^\natural \times^\Gamma T_{\gamma^{-1}} |\omega_{orb}| = \int^{\mathrm{mot}}_{Y^\natural} M^\natural |\bar{\lambda}_\gamma^*\omega_{orb}|.  \]
By definition of $\bar{\lambda}_\gamma$ we have 
\[\bar{\lambda}_\gamma^*\omega_{orb} = t^{\sum_i \frac{c_i}{r}}\omega_{orb} = t^{w(\gamma)} \omega_{orb}.\]
Hence 
\[  \int^{\mathrm{mot}}_{Y^\natural} M^\natural |\bar{\lambda}_\gamma^*\omega_{orb}| =  \int^{\mathrm{mot}}_{Y^\natural} M^\natural |t^{w(\gamma)} \omega_{orb}| = \BL^{-w(\gamma)}  \int^{\mathrm{mot}}_{Y^\natural} M^\natural |\omega_{orb}|. \]
So it remains to prove that 
\[\int^{\mathrm{mot}}_{Y^\natural} M^\natural |\omega_{orb}| = 1.\]

Let $M^{\circ}_k$ be the complement in $M_k$ of the coordinate hyperplanes.
Consider the morphism $\ac : \underline{M/\Gamma} \to \underline{M_k / \Gamma}_{{}}$.
We set $Y^\sharp = Y^\natural \cap (\ac^{-1} ( \underline{M^{\circ}_k / \Gamma}_{{}}))$ and  $M^\sharp = M^\natural \cap (\ac^{-1} (\underline{M^{\circ}_k}_{{}}))$.
Since $\ac^{-1}(\underline{M_k}_{{}} \setminus \underline{M^{\circ}_k}_{{}})$ is of dimension $<d$,
\[\int^{\mathrm{mot}}_{Y^\natural} M^\natural |\omega_{orb}|
=
\int^{\mathrm{mot}}_{Y^\sharp} M^\sharp |\omega_{orb}|.\]
Note that $M^\sharp$ is definably isomorphic to
$\{(y, x) \in Y^\sharp \times \underline{M_k}_{{}} \ | \ \ac (y) = p (x)\}$,
with $p$ the projection $M^{\circ}_k \to M^{\circ}_k / \Gamma$.
Indeed, any point $z$ in  $M^\sharp$ is uniquely determined by its image in
$Y^\sharp$ and $\ac(z)$.

Thus, if $\varphi$ denotes the class of $p: M^{\circ}_k \to M^{\circ}_k / \Gamma$
in $\cC_\mot^\Gamma( \underline{M^{\circ}_k / \Gamma})$,
\[\int^{\mathrm{mot}}_{Y^\sharp} M^\sharp |\omega_{orb}|
=
\int^{\mathrm{mot}}_{Y^\sharp} \ac^* (\varphi) |\omega_{orb}|
.\]

Now let $y \in \ac (Y^\sharp)$ and
$x \in M^{\circ}_k$ with $p (x) = y$.
We have
\[
\int^{\mathrm{mot}}_{\ac^{-1} (y) \cap Y^\sharp}  |\omega_{orb}|
=
\int^{\mathrm{mot}}_{\ac^{-1} (x) \cap M^\sharp}  |\omega_{M}|
=
\int^{\mathrm{mot}}_{\ac^{-1} (x)}  |\omega_{M}| = (\mathbb{L} - 1)^{-d}.\]

By Proposition \ref{eqmotfubini} (Fubini), we deduce that
\[
\int^{\mathrm{mot}}_{Y^\sharp} \ac^* (\varphi) |\omega_{orb}|
=
(\mathbb{L} - 1)^{-d} \int^{\mathrm{mot}}_{\underline{M^{\circ}_k / \Gamma}_{{}}} [M^{\circ}_k]
.\]
So, to conclude it is enough to verify that
the class of $M^{\circ}_k$  in $\cC^\Gamma_\mot(\ast_{k})$
is equal  to $(\mathbb{L} - 1)^{d}$, with trivial $\Gamma$-action.
This follows from
Lemma \ref{isomot}
for the isogeny
$M^{\circ}_k \to M^{\circ}_k / \Gamma$, since $M^{\circ}_k \cong \mathbb{G}_{m,k}^d$.
\end{proof}

%\begin{lemma} There is a well defined $k\llb t \rrb$-morphism
%\[ \lambda_\gamma: \BA^n \to \BA^n \times^{\Gamma} T_\gamma \cong \Spec\left(k[[t^{1/d}]](x_1,\dots,x_n)^\gamma\right),\]
%given on the level of coordinate rings by
%\[ f(x_1,\dots,x_n) \mapsto f(t^\frac{c_1}{|\Gamma|}x_1,\dots,t^\frac{c_n}{|\Gamma|}x_n). \]
%Furthermore $\lambda_\gamma$ is an isomorphism over $k\llp t \rrp$.
%\end{lemma}
%\begin{proof} Since the statement depends only on the subgroup generated by $\gamma$ we may assume $\Gamma = \left\langle %\gamma\right\rangle$, in particular $|\Gamma|=d$. Let $f\in k[[t^{1/d}]](x_1,\dots,x_n)$ be invariant under $\gamma$ and write 
%\[ f(x_1,\dots,x_n) = \sum_{\alpha \in \BZ^n_{\geq 0}} f_\alpha(t^{1/d}) x^\alpha,    \]
%where we used the notation $x^\alpha = x_1^{\alpha_1} \dots x_n^{\alpha_n}$. Since $\gamma$ preserves the degree of $x^\alpha$ we must have for all $\alpha$
%\[f_\alpha(t^{1/d}) x^\alpha = \gamma(f_\alpha(t^{1/d}) x^\alpha) = \xi^{\sum_i \alpha_ic_i} f_\alpha(\xi t^{1/d})x^\alpha. \]
%Hence we must have $f_\alpha \in t^{-\frac{1}{d}\sum_i \alpha_ic_i}k\llb t \rrb$, which implies that $\lambda_\gamma(f) \in k\llb t \rrb(x_1,\dots,x_n)$.

%Over $k\llp t \rrp$ we can invert $t$ to define an inverse morphism to $\lambda_\gamma$.
%\end{proof}

Next we go back to the situation of Theorem \ref{equivorb} but considering only a global fixed point.

\begin{proposition}\label{globalcyclic} Let $x\in  |M_k|$ be fixed by $\Gamma$. Then for any $\gamma \in \Gamma$ we have
\[   \int^{\mathrm{mot}}_{e^{-1}(x)} M^\natural \times^\Gamma T_{\gamma^{-1}} |\omega_{orb}| =   \BL^{-w(\gamma)} \in \cC_\mot^\Gamma(\ast_{k(x)}).  \]
In particular, 
for any $\delta \in \widehat{\Gamma}$ which is non-trivial,
\[ \int^{\mathrm{mot}, \delta}_{e^{-1}(x)} M^\natural\times^\Gamma T_{\gamma^{-1}} |\omega_{orb}| = 0.\] 
\end{proposition}
\begin{proof}After base change we may assume $x$ is a $k$-rational point.
By \cite[Lemma 4.14]{gwz} there exists a $\Gamma$-invariant open neighbourhood $U \subset M$ of $x$ and an \'etale morphism 
\[f:U \rightarrow \BA^d_{k\llb t \rrb},\]
which is $\Gamma$-invariant for some linear action of $\Gamma$ on $\BA^d_{k\llb t \rrb}$ and sends $x$ to the origin (in \textit{loc. cit.} the lemma is only proven for non-archimedean local fields but the proof applies to any Henselian ring). Furthermore $f$ induces a definable bijection
\[ \bar{f}: e^{-1}(x) \rightarrow e^{-1}_{\BA^d_{k\llb t \rrb}}(0).  \]
Since $f$ is \'etale we have $|f^* \omega_{orb,\BA^n_{k\llb t \rrb}}| = |\omega_{orb,M}|$ and thus the proposition follows from Proposition \ref{linearaffine}.
\end{proof}

Finally we can prove the general point-wise statement.

\begin{proposition}\label{pointwise} For any point $x \in |M_k^\gamma/\Gamma|$ we have 
\[    \int^{\mathrm{mot}}_{e^{-1}(x)} M^\natural \times^\Gamma T_{\gamma^{-1}} |\omega_{orb}| =   \BL^{-w(\gamma)}[\pi^{-1}(x)]  \]
in $\cC_\mot^\Gamma(\ast_{k (x)})$,   %tors
where $\pi: M^\gamma_k \rightarrow M^\gamma_k/\Gamma$ denotes the quotient morphism.
In particular, for any $\delta \in \widehat{\Gamma}$,
\[    \int^{\mathrm{mot}, \delta }_{e^{-1}(x)} M^\natural \times^\Gamma T_{\gamma^{-1}} |\omega_{orb}| =   \BL^{-w(\gamma)}\vartheta_{\ast_{k (x)}}^{\delta} ([\pi^{-1}(x)])  \]
in $\cC_\mot(\ast_{k (x)})$. %tors
\end{proposition}

\begin{proof}After base change we may assume $x$ is a $k$-rational point. Let $\Gamma' \subset \Gamma$ denote the stabilizer group of $x$. We denote by $e_{\Gamma'},\omega_{orb,\Gamma'}$ the specialization morphism and orbifold form for $M/\Gamma'$. We have a commutative diagram

\begin{equation}\label{ecom}\xymatrix{ M^\natural/\Gamma' \ar[r]^<<<<{e_{\Gamma'}} \ar[d] &  \bigsqcup_{\delta \in \Gamma'} \underline{M_k^\delta/\Gamma' }_{{}} \ar[d] \\
M^\natural/\Gamma \ar[r]^<<<<<e & \bigsqcup_{\delta \in \Gamma} \underline{M_k^\delta/\Gamma}_{{}},}\end{equation}
where the vertical arrows are quotient morphisms.

We assume first that $x$ lifts to a $k$-rational point $\tilde{x} \in M_k^\gamma$, in which case we have $\pi^{-1}(x) \cong \Gamma/\Gamma'$. Then we have an identification
\[ e^{-1}_{\Gamma'}(\pi^{-1}(x)) = \bigsqcup_{x'\in\pi^{-1}(x)} e^{-1}_{\Gamma'}(x'), \]
and the free $\Gamma/\Gamma'$ action on $e^{-1}_{\Gamma'}(\pi^{-1}(x))$ simply permutes the component on the right hand side. Therefore we have an injective morphism
\[\sigma: e^{-1}_{\Gamma'}(\tilde{x}) \rightarrow \bigsqcup_{x'\in\pi^{-1}(x)} e^{-1}_{\Gamma'}(x') \rightarrow e^{-1}(x),\]
 which is in fact an isomorphism. Indeed for any algebraically closed field extension $K/k$, any $K$-point $y: \Spec(K\llp t\rrp) \rightarrow M/\Gamma$ in $e^{-1}(x)$ lifts by definition of a $\Gamma$-equivariant morphism $T_\gamma \rightarrow M$. As $\gamma \in \Gamma'$ the quotient $T_\gamma /\Gamma'$ is the trivial $\Gamma/\Gamma'$-torsor, which gives a lift of $y$ to $e^{-1}_{\Gamma'}(\tilde{x})$.

From the isomorphism $M^\natural \times_{M^\natural/\Gamma} M^\natural \cong M^\natural \times \Gamma$ we see that 

\[\left(M^\natural \times^\Gamma T_{\gamma^{-1}}\right)\times_{M^\natural/\Gamma} M^\natural/\Gamma' \cong \left(M^\natural \times^\Gamma T_{\gamma^{-1}}\right)\times^{\Gamma'} \Gamma .\]

Hence we have

\begin{align*}\int^{\mathrm{mot}}_{e^{-1}(x)} M^\natural \times^\Gamma T_{\gamma^{-1}} |\omega_{orb}| &= \int^{\mathrm{mot}}_{e_{\Gamma'}^{-1}(\tilde{x})} \sigma^*\left(M^\natural \times^\Gamma T_{\gamma^{-1}}\right) |\sigma^*\omega_{orb}|\\
&= \int^{\mathrm{mot}}_{e_{\Gamma'}^{-1}(\tilde{x})} \left(M^\natural \times^\Gamma T_{\gamma^{-1}}\right)\times^{\Gamma'} \Gamma |\omega_{orb,\Gamma'}|=\BL^{-w(\gamma)}[\Gamma/\Gamma']. \end{align*}
For the last equation we used Lemma \ref{unrtw} and Proposition \ref{globalcyclic}.

If $x \in M^\gamma/\Gamma(k)$ does not lift of a $k$-point of $M^\gamma$ we can find a $\Gamma$-torsor $T$ over $\Spec(k)$, such that $x$ lifts to a $k$-point $\tilde{x}$ in $M^\gamma \times^\Gamma T$, in which case we have $\pi^{-1}(x) \cong \Gamma/\Stab(\tilde{x}) \times^\Gamma T^{-1}$. We can consider $T$ as an unramified $k\llb t \rrb$-torsor and use Lemma \ref{unrtw} to write
\[ \int^{\mathrm{mot}}_{e^{-1}(x)} M^\natural \times^\Gamma T_{\gamma^{-1}} |\omega_{orb}| = \left(\int^{\mathrm{mot}}_{e^{-1}(x)} (M^\natural \times^\Gamma T) \times^\Gamma T_{\gamma^{-1}} |\omega_{orb}|\right) \times^\Gamma T^{-1}. \]
Repeating the previous argument with $M^\natural$ replaced by $M^\natural \times^\Gamma T$ then proves the proposition.
\end{proof}

\begin{proof}[Proof of Theorem \ref{equivorb}] We can see the equality 
\[\int_{e^{-1}(\gamma)} M^\natural \times^{\Gamma} T_{\gamma^{-1}} |\omega_{orb}| = \BL^{-w(\gamma)} [M^\gamma], \]
as the push-forward to the point of an equality in $\cC_{\mot}^\Gamma(M^\gamma/\Gamma)$. %top
The equality can thus be %\todo{Should be cite Raf and Immi here?}  FL ce n'est pas nécessaire
checked pointwise for every $x \in |M^\gamma/\Gamma|$, which is exactly Proposition \ref{pointwise}.
\end{proof}

\subsection{$E$-polynomials and stringy invariants}

In this subsection let $k= \BC$.
\subsubsection{$E$-polynomials of varieties} Recall for example from \cite{Batyrev1996} that one has a realization homomorphism 
\[ E:K_0 (\Var_k) \longrightarrow \BZ[x,y] \]
given by the $E$-polynomial. For an algebraic variety $M$ over $k$ it is defined by the Hodge-Deligne polynomial
\[ E(M;x,y) = \sum_{i,p,q\geq 0} (-1)^i h_c^{i;p,q}(X) x^py^q,  \]
where the $h_c^{i;p,q}(M)$ are the compactly supported mixed Hodge numbers of $M$. 

By the functoriality of Deligne's mixed Hodge structures \cite{De71} one can do the same construction equivariantly. This way one obtains for a finite abelian group $\Gamma$ a morphism
\[E^\Gamma: K_0 (\Var^\Gamma_k) \rightarrow R_\Gamma[x,y],   \]  
where $R_\Gamma$ denotes the (complex) representation ring of $\Gamma$, that is the group ring of the character group $\widehat{\Gamma}$. Furthermore $E^\Gamma$ admits a decomposition into isotypical components
\[ E^\Gamma = \bigoplus_{\varrho \in \widehat{\Gamma}} E^\varrho. \]
Note however that for $\varrho \not\equiv 1$, $E^\varrho$ is only additive and not multiplicative.

Next we recall the definition of the stringy $E$-polynomial. We restrict ourselves to the special case, where $M$ is a smooth $k$-variety with an action of a finite abelian group $\Gamma$ preserving the canonical bundle of $M$. As before we also assume that the weight function $w$ is constant on $M^\gamma$ for each $\gamma \in \Gamma$. Then we define the stringy $E$-polynomial of the quotient stack $\mathbb{X}=[M/\Gamma]$ by
\[ E_{\st}(\BX;x,y) = \sum_{\gamma \in \Gamma} (xy)^{\dim M-w(\gamma)} E(M^\gamma/\Gamma;x,y). \]

We will also need a twisted version $E^\varrho_{\st}$ of $E_{\st}$ for a class $\varrho \in H_{\rm grp}^2(\Gamma,\mu_n)$, where $H_{\rm grp}^2$ denotes group cohomology. Such a class defines for every $\gamma \in \Gamma$ a character $\varrho_\gamma: \Gamma \to \mu_n$ and we set
\[E_{\st}^\varrho(\BX;x,y) = \sum_{\gamma \in \Gamma} (xy)^{\dim M-w(\gamma)} E^{\varrho_\gamma}(M^\gamma;x,y).\]
\begin{rmk} The notation $E_{\st}^\varrho(\BX;x,y)$ is slightly misleading, as the polynomial depends not only on $\BX$ but rather on $M$ as a $\Gamma$-variety.
\end{rmk}

\subsubsection{$E$-polynomials of Chow motives}\label{Epol}
The $E$-polynomial factorizes through the Grothen\-dieck ring of Chow motives.
More precisely, there exists a Hodge realization functor
\[\mathrm{DM_{gm}}(k,\mathbb{Q})
\longrightarrow
D^b (\mathrm{MHS}^p_{\mathbb{Q}})
\]
with values in the bounded derived category of polarized mixed $\mathbb{Q}$-Hodge structures, cf. 
\cite{levine}\cite{huber},
which extends to a functor
\[\mathrm{DM_{gm}}(k,\Lambda)
\longrightarrow
D^b (\mathrm{MHS}^p_{\mathbb{C}})
\]
with values in the bounded derived category of polarized mixed $\mathbb{C}$-Hodge structures.
Since, as recalled in
Remark \ref{bonda},
$K_0 (\mathrm{DM_{gm}}(k,\Lambda))$ is isomorphic to
$K_0 (\mathrm{M_{rat}} (k, \Lambda))$, it follows that the $E$-polynomial
factorizes through a morphism
\[ E:K_0 (\mathrm{M_{rat}} (k, \Lambda)) \longrightarrow \BZ[x,y]. \]

Now if $X$ is a reduced and separated $k$-scheme of finite type endowed with an action of a finite abelian group $\Gamma$
and $\varrho$ is a complex character of $\Gamma$, we have that
$E^{\varrho} ([X]) = E (\chi^{\varrho}_{\ast, c} (X))$, using notations from \ref{equicase}.

\section{N\'eron Models, isogenies and pairings}\label{sec4} 

\subsection{N\'eron Models}Let $k$ be a field of characteristic $0$ and $\FP$ an abelian variety over $k\llp t \rrp$. The N\'eron model $\FN(\FP)$ of $\FP$ is the unique smooth, separated and finite type group scheme over $k \llb t\rrb$, such that for any smooth $k \llb t\rrb$-scheme $X$ and any morphism $X_{k\llp t\rrp} \rightarrow \FP$ there exists a unique extension $X \rightarrow \FN(\FP)$. In particular we have for each extension $K/k$ a bijection 
\[\FP(K\llp t\rrp) \xrightarrow{\sim} \FN(\FP)(K\llb t\rrb).\]  

This implies that the assignments $\underline{\FP}$ and $\underline{\FN(\FP)}_{\circ}$ are isomorphic, which is why N\'eron models appear naturally in the context of motivic integration,
as first observed in \cite{LS}.  
The following remark makes this more concrete.

\begin{rmk}\label{nerint} Since the canonical bundle of $\FP$ is trivial and $\FP$ is proper, we can choose a global translation-invariant volume forms $\omega$ on $\FP$, which is unique up to a scalar in $k\llp t\rrp$. If $\ord_{\FN}\omega \in \BZ$ denotes the order of vanishing of $\omega$ along the special fiber $ \FN(\FP)$, we have that $t^{-\ord_{\FN}\omega}\cdot \omega$ is a global volume form on $\FN(\FP)$. 
By Lemma \ref{reduc} together with Proposition 12.6 of \cite{clcrelle}, we have
\begin{equation}\label{abvoln} \int_{\underline{\FP}} |\omega| =\BL^{-\ord_{\FN}\omega} \int_{\underline{\FN(\FP)}_{\circ}} |t^{-\ord_{\FN}\omega}\cdot\omega| =  
\BL^{-d-\ord_{\FN}\omega}[\FN(\FP)_k],\end{equation}
with $d$ the dimension of $\FP$.
Indeed, Lemma \ref{reduc} implies the second equality, while Proposition 12.6 of \cite{clcrelle} implies the equality between the left hand side of the first equation and the right hand side of the last equation.
\end{rmk}

Let $\FN^0(\FP)$ denote the fiberwise identity component of $\FN(\FP)$. The quotient 
\[\Phi_{\FP} = \frac{\FN(\FP)_k}{\FN^0(\FP)_k}\]
is a finite \'etale $k$-group scheme called the component group of $\FN(\FP)$. If $k$ is algebraically closed we will consider $\Phi_{\FP}$ simply as a finite group. 

As for finite groups in Remark \ref{compdir} we call a $\FP$-torsor $T$ unramified, if $T$ becomes trivial over $\overline{k} \llp t\rrp$ i.e. if $T(\overline{k} \llp t\rrp) \neq \emptyset$. It then follows from \cite[Corollary 6.5.3]{NeronModels} that $T$ extends uniquely to a $\FN(\FP)$-torsor over $k\llb t\rrb$  which we denote by $\FN(T)$.

\subsection{Self-dual isogenies}\label{sdis} In this subsection we assume $k$ to be algebraically closed of characteristic $0$. Let $\FP$ and $\FP'$ be  abelian varieties over $k\llp t \rrp$ and 
\[\phi: \FP \longrightarrow \FP'\]
be an isogeny with kernel $\Gamma$.  By \cite[Proposition 7.3.6]{NeronModels} $\phi$ extends to an isogeny of N\'eron models $\FN(\phi): \FN(\FP) \rightarrow \FN(\FP')$, meaning that $\FN(\phi)$ is fiberwise finite and surjective on identity components. The kernel $\underline{\Gamma} = \ker(\FN(\phi))$ is an \'etale group scheme over $k\llb t\rrb$, in particular we have an identification of finite groups
\[\Gamma(k\llp t\rrp) = \underline{\Gamma}(k\llb t\rrb) = \underline{\Gamma}_k(k). \]
Notice however, that in general $\Gamma(\overline{k\llp t\rrp}) \neq \Gamma(k\llp t\rrp)$ as $\FN(\phi)$ is only quasi-finite. The following two finite groups will be interesting for us later 
\[ \Gamma^0 = \underline{\Gamma}(k\llb t\rrb) \cap \FN^0(\FP)(k\llb t\rrb), \ \ \ \ \ \ \Gamma' = \underline{\Gamma}(k\llb t\rrb) /\Gamma^0.   \]

%\begin{lemma}\label{nrestr} The natural restriction map $\res: A(F) = \FA(\FO) \to \FA(k)$ induces a bijection on torsion points. Furthermore $\ker(\res)$ is uniquely divisible.
%\end{lemma}
%\begin{proof} Let $m$ be any integer. Then multiplication by $m$ defines an \'etale morphism $[m]:\FA \to \FA$, hence the group scheme of $m$-torsion points $\FA[m]$ is \'etale and we have $\FA[m](\FO) = \FA[m](k)$ by the formal lifting property of \'etale morphisms.

%The statement about $\ker(\res)$ follows similarly by considering the \'etale morphism $[m] - a: \FA \to \FA$ for any $a \in \ker(\res)(\FO)$.
%\end{proof}

%Let $\alpha:A \to B$ be an isogeny of abelian varieties over $F$ and set $\Gamma = \ker(\alpha)$. Then by \cite[Proposition 7.3.6]{NeronModels} $\alpha$ extends to an isogeny of N\'eron models $\alpha: \FA \to \FB$, which is in particular \'etale. It follows that $\Gamma$ extends to an quasi-finite \'etale group scheme over $\FO$. We define $\Gamma^0 = \Gamma \cap \FA^0$ and $\Gamma' = \Gamma /\Gamma^0$.

The short exact sequence of abelian group schemes over $k\llp t\rrp$
\begin{equation}\label{sses}  0 \to \Gamma \to \FP \to \FP' \to 0\end{equation}
gives rise to a long exact sequence of \'etale cohomology groups
\[0 \to \Gamma(k\llp t\rrp) \to \FP(k\llp t\rrp) \xrightarrow{\phi} \FP'(k\llp t\rrp) \xrightarrow{\partial_\phi} H^1(k\llp t\rrp,\Gamma) \to H^1(k\llp t\rrp,\FP) \to \dots\]

The image of $\partial_\phi$ admits a description in terms of N\'eron models by means of the following commutative diagram with exact rows:

\begin{equation}\label{bigdia} \xymatrix{  0 \ar[r] & \Gamma(k\llp t\rrp) \ar[r] \ar[d]^{\cong} & \FP(k\llp t\rrp) \ar[r] \ar[d] & \FP'(k\llp t\rrp) \ar[r]^{\partial_\phi} \ar[d]& \im(\partial_\phi) \ar[r] \ar[d] & 0 \\
							 0 \ar[r] & \underline{\Gamma}_(k) \ar[r] \ar[d] & \FN(\FP)(k) \ar[r] \ar[d]& \FN(\FP')(k) \ar[r] \ar[d]& \FN(\FP')(k) / \FN(\FP)(k) \ar[r] \ar[d]& 0 \\
							 0 \ar[r] & \Gamma' \ar[r] & \Phi_{\FP} \ar[r] & \Phi_{\FP'} \ar[r] & \Phi_{\FP'}/\Phi_{\FP} \ar[r] & 0.
}\end{equation}
Here we write $ \FN(\FP)(k)$ for $ \FN(\FP)_k(k)$ and $\FP(k\llp t\rrp) \to \FN(\FP)(k)$ for the composition $\FP(k\llp t\rrp) \cong \FN(\FP)(k \llb t \rrb) \to \FN(\FP)(k)$, and similar for $\FP'$. We claim that both arrows in the composition 
\begin{equation}\label{imdesc}   \im(\partial_\phi) \rightarrow \FN(\FP')(k) / \FN(\FP)(k) \rightarrow \Phi_{\FP}/\Phi_{\FP'} \end{equation}
are isomorphisms. Indeed, for the first arrow we notice that 
\[\ker\left(\FP(k\llp t\rrp) \to \FP(k)\right) \cong \ker \left(\FP'(k\llp t\rrp) \to \FP'(k)\right)
\] as the restriction $\FP(k\llp t\rrp) \to \FP(k)$ induces a bijection on torsion points \cite[Proposition 7.3.3]{NeronModels}. The second arrow is an isomorphism since the map $\FN^0(\FP)(k) \rightarrow \FN^0(\FP')(k)$ is surjective.

\begin{lemma}\label{neutc} Let $m = |\Phi_\FP|$, $m' = |\Phi_{\FP'}|$ and $\phi^{r}$ the composition $\phi \circ [r]$ for any positive integer $r$. Then we have  
\[ \im(\partial_\phi) = \ker\left( H^1(k\llp t\rrp,\Gamma) \rightarrow H^1(k\llp t\rrp,\ker(\phi^{m'}))\right) \]
and
\[ \Gamma^0 = \im\left(\ker(\phi^m)(k\llp t\rrp) \xrightarrow{[m]} \FP[n](k\llp t\rrp)\right).     \]
\end{lemma}

\begin{proof} Consider the commuting diagram

\[ \xymatrix{ \FP(k\llp t\rrp) \ar[r]^{\phi} \ar[d]^{=} & \FP'(k\llp t\rrp) \ar[r]^{\partial_\phi} \ar[d]^{[m']} & H^1(k\llp t\rrp,\Gamma) \ar[d] \\
		\FP(k\llp t\rrp) \ar[r]^{\phi^{m'}} &	\FP'(k\llp t\rrp) \ar[r]^<<<<<<<{\partial_{\phi^{m'}} \ \ \ } & H^1(k\llp t\rrp,\ker(\phi^{m'})).
}\]
By construction $\im([m']) \subset \FN^0(\FP')(k\llb t\rrb)$ and since $\FN^0(\FP')(k\llb t\rrb)$ is divisible we deduce $\partial_{\phi^{m'}} \circ [m'] = 0$. This implies the first assertion. The second one follows directly from the divisibility of $\FN^0(\FP)(k\llb t\rrb)$.
\end{proof}

Finally we consider the case where $\FP' = \FPh$ is the dual abelian variety of $\FP$ and the isogeny $\phi$ is selfdual i.e. $\widehat{\phi} = \phi$. In this case dualizing the short exact sequence \eqref{sses} gives an isomorphism of $\Gamma$ with its Cartier dual, which in turn defines a non-degenerate pairing
\[ \langle \cdot, \cdot \rangle_{\phi}: \Gamma \times \Gamma \rightarrow \mu_n,   \]
compatible with the action of $\Gal(k\llp t\rrp) = \widehat{\mu}$ on $\Gamma$. Here $n$ denotes the order of $\Gamma$. We will use the same notation for the pairing 
\begin{equation}\label{cpp} \langle  \cdot, \cdot \rangle_{\phi}: \Gamma(k\llp t\rrp) \times H^1(k\llp t\rrp,\Gamma) \rightarrow H^1(k\llp t\rrp, \mu_n) = \Hom(\widehat{\mu},\mu_n)= \BZ/n\BZ,  \end{equation}
induced by taking cup-products. As in \cite[Proposition 3]{Og62} this pairing is again non-degenerate.

\begin{lemma}\label{orthor} With respect to the pairing \eqref{cpp} the groups $\Gamma^0 \subset \Gamma(k\llp t\rrp)$ and $\im(\partial_\phi) \subset H^1(k\llp t\rrp,\Gamma) $ are exact anihilators of each other.
\end{lemma}

\begin{proof} The Grothendieck pairing on component groups gives a perfect pairing between $\Phi_{\FP}$ and $\Phi_{\FPh}$ \cite[Expos\'e IX]{GRR67}, in particular $|\Phi_{\FP}|=|\Phi_{\FPh}|$. From \eqref{bigdia} and \eqref{imdesc} we thus see 
\[| \im(\partial_\phi)| |\Gamma^0| = |\Gamma(k\llp t\rrp)|.\]
Hence it is enough to show that $\langle \cdot, \cdot \rangle_\phi$ vanishes on $\Gamma^0 \times \im(\partial_\phi)$. For this let $m= |\Phi_{\FP}|=|\Phi_{\FPh}|$ and $\phi^m= \phi \circ [m]$ as in Lemma \ref{neutc}. Then $\phi^m$ is again self-dual and essentially by construction we have the following compatibilty 
\[  \langle x^m,y \rangle_{\phi} = \langle x, i(y)\rangle_{\phi^m},\]
where $x \in \ker(\phi^m)(k\llp t\rrp)$, $y \in H^1(k\llp t\rrp,\Gamma)$ and $i:H^1(k\llp t\rrp,\Gamma) \rightarrow H^1(k\llp t\rrp,\ker(\phi^m))$.
The lemma now follows directly from Lemma \ref{neutc}
\end{proof}

Assume now that $\underline{\Gamma}$ is constant over $k\llp t\rrp$, which will be the only case of interest for us. In this case we identify $\Gamma$ with its group of $k$-points. We further have an identification $ H^1(k\llp t\rrp,\Gamma) \cong \Gamma$ given by the choice of a primitive $|\Gamma|$-th root of unity, see Remark \ref{compdir}. The following corollary will be crucial for us.

\begin{corollary}\label{csum} 
 For any $\gamma \notin \im(\partial_\phi)$ we have
\[\chi_{*,c}^{\varrho_\gamma} ( [\FN(\FP)_k]) = 0. \]
\end{corollary}
\begin{proof} As $\FN^0(\FP)_k$ is connected, the action of the finite group $\Gamma^0$ on  the homological Chow motive $M (\FN(\FP)_k)$ of $\FN(\FP)_k$
 is trivial
by Corollary \ref{autab}.
On the other hand,  Lemma \ref{orthor} implies that $\varrho_\gamma$ is non-trivial on $\Gamma^0$ and thus 
$\chi_{*,c}^{\varrho_\gamma} ( [\FN(\FP)_k]) = 0$.
\end{proof}

\section{The Hausel-Thaddeus conjecture and a motivic version of  topological mirror symmetry}\label{sec5}

In this section we recall some basic facts about Higgs bundles and state the topological mirror symmetry conjecture of Hausel-Thaddeus. For a more detailed account we refer to their paper \cite{MR1990670} and the references therein. 

\subsection{Higgs bundles}Let $C$ be a connected, smooth, projective curve of genus $g \geq 2$ over a field $k$ of characteristic $0$ . A \textit{Higgs bundle} on $C$ is a pair $(E,\theta)$ consisting of a vector bundle $E \to C$ and a twisted endomorphism 
\[\theta: E \to E\otimes K_C,\]
where $K_C$ denotes the canonical bundle of $C$. 

\begin{definition}\label{modsp}  Let $L \to C$ be a line bundle of degree $d \in \BZ$. \begin{enumerate} 
\item We denote by $\M_n^L$  the moduli space of semi-stable $L$-twisted $\mathrm{SL}_n$-Higgs bundles, that is Higgs bundles $(E,\theta)$ of rank $n$ on $C$ together with an isomorphism $\det E \cong L$ satisfying $\tr (\theta) = 0 \in H^0(C,K_C)$.
\item The finite group scheme $\Gamma = \Pic(C)[n] \equiv (\BZ/n\BZ)^{2g}$ acts on $ \M_n^L$ by tensoring the underlying vector bundle of a Higgs field. We write $\widehat{\BM}_n^d= [\M_n^L / \Gamma]$ and $\widehat{\M}_n^d = \M_n^L / \Gamma$ for the stack resp. geometric quotient. 
\end{enumerate}
\end{definition}

\begin{rmk} $\widehat{\BM}_n^d$ and $\widehat{\M}_n^d$ can be identified with a suitably defined moduli  stack resp. space of semi-stable $\PGL_n$ Higgs-bundles of degree $d$ \cite[Section 6]{GH13},  in particular they depend only on $d$ and not on $L$.
\end{rmk}

We will only consider the case where $n$ and $d$ are coprime, in which case every semi-stable Higgs bundle is stable and $\M_n^L$ is a smooth quasi-projective variety. 

\subsection{Hitchin fibrations and duality}\label{fibnot}

To a Higgs bundle $(E,\theta)$ one can associate its characteristic polynomial, whose $i$-th coefficient is given by $(-1)^i\tr(\wedge^i \theta) \in H^0(C,K_C^i)$. This defines morphisms
\[\xymatrix{ \M^L_n \ar[dr]^{h} &   &\widehat{\M}^d_{n} \ar[dl]_{\hat{h}}\\
				& \A = \bigoplus_{i=2}^n H^0(C,K^{\otimes i}).
}
\]
For $a \in \A$ we write $\M^L_{n,a}$ and $\widehat{\M}^d_{n,a}$ for the fibers $h^{-1}(a)$ and $\hat{h}^{-1}(a)$ respectively. To describe the generic fibers of $h$ and $\hat{h}$ we use the spectral curve construction. For any $a=(a_i)_{2\leq i \leq n} \in \A$ the spectral curve $C_a$ is the subscheme inside the total space of $K_C$ defined by the equation 
\[ \{X^n + a_2 X^{n-2}+ \dots + a_n = 0\} \subset \text{Tot}(K_C).\]
There is a non-empty open subvariety $\A^{sm} \subset \A$ defined by the condition that $C_a$ is smooth and geometrically connected. For $a \in \A^{sm}$ the projection $\pi: C_a \to C$ is a ramified cover of degree $n$ and hence induces a degree preserving norm map
\[ \Nm: \Pic(C_a) \to \Pic(C).  \]
The Prym variety $\P_a$ is defined as the kernel of $\Nm$. It is an abelian variety and its dual $\widehat{\P}_a$ can be identified with the quotient $\P_a/ \Gamma$, where the inclusion $\pi^*: \Gamma \to \P_a$ is defined by pullback along $\pi: C_a \to C$. We quickly explain this duality, see also \cite[Lemma 2.3]{MR1990670}. By definition there is a short exact sequence

\[  0 \longrightarrow \P_a \longrightarrow   \Pic^0(C_a) \xrightarrow{\Nm} \Pic^0(C) \longrightarrow 0.   \]
Dualizing the sequence and using the auto-duality of $\Pic^0$ we get

\begin{equation}\label{dualseq} 0 \longrightarrow \Pic^0(C) \xrightarrow{\pi^*}   \Pic^0(C_a) \longrightarrow \widehat{\P}_a  \longrightarrow 0.\end{equation}
Finally there is an isomorphism $\frac{\Pic^0(C_a)}{\Pic^0(C)} \cong \P_a / \Gamma$ induced by the following morphism
\begin{align*}    \Pic^0(C_a) &\longrightarrow \P_a / \Gamma \\          
 M & \longmapsto M \otimes \pi^* \Nm(M^{-1})^\frac{1}{n},     \end{align*}
where $(\cdot)^\frac{1}{n}: \Pic^0(C) \xrightarrow{\sim} \Pic^0(C)/ \Gamma$ denotes the isomorphism induced by the $n$-th power map. In particular the quotient $\P_a \rightarrow \P_a/ \Gamma$ factors as 
\begin{equation}\label{selfd} \P_a \rightarrow \Pic^0(C_a) \rightarrow \P_a /\Gamma,  \end{equation}
and is therefore selfdual.

\begin{proposition}\cite[Proposition 3.6]{MR998478}\label{torfib} For any $a \in \A^{sm}$ we have isomorphisms 
\[ \M^L_{n,a} \cong (\Nm)^{-1}(L) \ \ \text{  and  } \ \  \widehat{\M}^d_{n,a} \cong (\Nm)^{-1}(L)/ \Gamma \cong \Pic^d(C_a)/\Pic^0(C).\]
In particular $\M^L_{n,a} $ and $\widehat{\M}^d_{n,a}$ are torsors for $\P_a$ and $\widehat{\P}_a$ respectively.
\end{proposition}
In \cite[Theorem 3.7]{MR1990670} Hausel and Thaddeus extend the duality between $\P_a$ and $\widehat{\P}_a$ to a duality between the torsors $\M^L_{n,a} $ and $\widehat{\M}^d_{n,a}$ and interpret this as (twisted) SYZ-mirror symmetry of $\M_n^L$ and $\widehat{\M}^{d'}_n$. This was their motivation for conjecturing Theorem \ref{htconj} below.

\begin{rmk}(Volume forms)\label{codim2} One can extend the the construction of the Prym variety to all of $\A$ and obtain a group scheme $\P \to \A$ acting on $\M^L_n$. The restriction $\M^{L,red}_n$ of $\M^L_n$ to $\A^{red} = \{a\in \A \ | \ C_a \text{ is reduced }\}$ contains a open $\widetilde{\M}^L_n \subset \M^{L,red}_n$ which is a torsor under $\P^{red}$ and dense in every Hitchin fiber over $\A^{red}$ \cite[Proposition 4.16.1]{MR2653248}. 
As $\mathrm{codim}\ \A \setminus \A^{red} \geq 2$ the same is true for $ \M^L_n \setminus \widetilde{\M}^L_n$ and we use this to construct an explicit $\Gamma$-invariant volume from $\omega^L$ on $ \M^L_n$ as follows.

Let $\omega_{\A}$ be a volume form on the affine space $\A$ and $\omega_P$ a translation-invariant trivializing section of the relative canonical bundle $K_{\P / \A}$. Then $\omega_{\P} \wedge \omega_{\A}$ is a global volume form on $\P$ which induces one on the torsor $\widetilde{\M}^L_n$ \cite[Lemma 6.13]{gwz}. As we have $\mathrm{codim}\ \M^L_n \setminus \widetilde{\M}^L_n \geq 2$ this form extends to a volume form $\omega^L$ on $\M^L_n$.
\end{rmk}

\subsection{The main result}
Let
\begin{equation}\label{weilp} \langle \cdot,\cdot \rangle: \Gamma \times \Gamma \longrightarrow \mu_n, \end{equation}
be the Weil pairing on $\Gamma = \Pic^0(C)[n]$. We write $\varrho \in H^2_{\rm grp}(\Gamma,\mu_n)$ for the class defined by $\langle \cdot,\cdot \rangle$ viewed as a $2$-cocycle and for $\gamma \in \Gamma$
\[ \varrho_\gamma = \langle \gamma, \cdot \rangle : \Gamma \to \mu_n\]
for the character induced by $\gamma$. 

The following theorem was conjectured by Hausel and Thaddeus in \cite[Conjecture 5.1]{MR1990670} and proven by Groechenig, Ziegler and the second author.

	\begin{theorem}\cite[Theorem 7.21]{gwz}\label{htconj} Assume $k= \BC$. Let $d,d'$ be integers prime to $n$ and $L,L'$  line bundles on $C$ of degree $d$ and $d'$ respectively. Let $q$ be the multiplicative inverse of $d'$ modulo $n$. Then we have
	\begin{equation}\label{htconjeq} E(\M_n^L;x,y) = E_{\st}^{\varrho^{-dq}}(\widehat{\BM}^{d'}_n;x,y). \end{equation}
\end{theorem}

\begin{rmks}\label{htconjrk}\begin{enumerate}
\item In the original conjecture the $E$-polynomial on the right hand side is twisted by a $\mu_n$-gerbe on $\widehat{\BM}^{d'}_n$, whereas we twist by the pullback of a $\mu_n$-gerbe on $[\Spec(k)/\Gamma]$. The equivalence of the two formulations is mentioned at the end of Section 4 in \cite{MR1990670} and proven in the case when $n$ is prime in Proposition 8.1 of \textit{loc. cit.} In the Appendix we explain how to deduce this equivalence in general. 
\item Our proof will show that both sides of \eqref{htconjeq} are independent of $L,L'$ and their degrees.
\item There is a natural $\BC^*$-action on $\M_n^L$ given by scaling the Higgs field, which commutes with the $\Gamma$-action. This gives gives $\M_n^L$ and $\M_n^{L,\gamma}$, for $\gamma \in \Gamma$, the structure of semi-projective varieties, see \cite{HV13}. In particular their cohomologies are pure \cite[Corollary 1.3.2]{HV13}. Hence the equality (\ref{tmsintrof}) stated in the introduction follows from 
(\ref{htconjeq}).
\end{enumerate}
\end{rmks}

We now formulate the main theorem of this article, an equality between virtual motives, which will imply Theorem \ref{htconj}.

\begin{theorem}\label{premainthm}  Let $k$ be an algebraically closed of characteristic $0$.
Let $d$, $d'$ and $q$ be integers prime to $n$ and let $L$ and $L'$ be line bundles on $C$ of degree $d$ and $d'$ respectively.  Then we have
\begin{equation}\label{premainthmeq} 
\BL^{-\dim \M^L_n} \vartheta_{\ast_{k}}([\M^L_n]) = \sum_{\gamma \in \Gamma}\BL^{ -w(\gamma)} \vartheta^{\varrho^{-q}_\gamma}_{\ast_{k}} ([\M_n^{L',\gamma}])
  \end{equation}
in $\cC_\mot (\ast_{k})$. 
\end{theorem}

In particular, when $L = L'$ and $q=1$ one recovers Theorem \ref{mainintro} from the Introduction.
Theorem \ref{htconj} follows from Theorem \ref{premainthm} by taking  the $E$-polynomial of both sides of 
(\ref{premainthmeq}).
Indeed,
$E (\vartheta_{\ast_{k}}([\M^L_n])) = E (\M^L_n; x, y)$
and 
$E (\vartheta^{\varrho^{-q}_\gamma}_{\ast_{k}} ([\M_n^{L',\gamma}])) =
E^{\varrho^{-q}_\gamma} (\M_n^{L',\gamma}; x, y)$, by the compatibility between (equivariant) $E$-polynomials of  varieties and
$E$-polynomials of their Chow motives recalled in \ref{Epol}.
%\end{proof}

\begin{rmk} The left hand side of \eqref{premainthmeq} admits a decomposition into $\Gamma$-isotypical components and our argument in Section \ref{sec6} in fact shows for any $\gamma \in \Gamma$
\[ \BL^{-\dim \M^L_n}  \vartheta^{\varrho^{-q}_\gamma}_{\ast_{k}}([\M^L_n]) = \BL^{ -w(\gamma)} \vartheta^{\varrho^{-q}_\gamma}_{\ast_{k}} ([\M_n^{L',\gamma}]). \]
On the level of $E$-polynomials, this refined version of topological mirror symmetry has been conjectured by Hausel (unpublished) and can be deduced by a Fourier-transform argument in the $p$-adic setting \cite[Theorem 7.24]{gwz}.
\end{rmk}

\subsection{Proof of Theorem \ref{premainthm}}
We will deduce Theorem \ref{premainthm} from an equality between motivic integrals that we will prove in the next section.

Fix $k$ an algebraically closed of characteristic $0$.
We will work on the constant pullbacks $\FM^L_n = \M^L_n \times_k \Spec(k\llb t \rrb)$ and $\widehat{\FM}^{d}_n = \widehat{\M}^{d}_n\times_k \Spec(k\llb t \rrb)=\FM^L_n/\Gamma$, which by functoriality parametrize twisted Higgs bundles on $\FC= C \times_k \Spec(k\llb t \rrb)$. Let
\[e_d:\widehat{\FM}^{d,\natural}_n \rightarrow \underline{I\M_n^L}_{{}} = \bigsqcup_{\gamma \in \Gamma} \underline{\M^{L,\gamma}_n/\Gamma}_{{}},   \]
be the specialization morphism from Construction \ref{econst} and 
\[\ooe_d:\widehat{\FM}^{d,\natural}_n \rightarrow \Gamma   \]
the composition of $e_d$ with the map to the index set of the disjoint union.

%We set for any $\gamma \in \Gamma$
%\[e_d^{-1}(\gamma) = e_d^{-1}( \underline{\M^{L,\gamma}_n/\Gamma}) \subset \widehat{\FM}^{d,\natural}_n.\]

Finally let $\omega^L$ be the $\Gamma$-invariant global volume form on $\FM^L_n$ from Remark \ref{codim2}  and $\omega^d_{orb}$ its quotient on $\widehat{\FM}^{d}_n$.

\begin{theorem}\label{mainthm}  Let $d,d'$ and $q$ be integers prime to $n$ and $L,L'$  line bundles on $C$ of degree $d$ and $d'$ respectively.  Then we have
\begin{equation}\label{mainthmeq} 
\int^{\mot}_{\underline{\FM^L_n}_{\circ}} |\omega^L| = \sum_{\gamma \in \Gamma} \int^{\mot,\varrho^{-q}_\gamma}_{\ooe_{d'}^{-1}(\gamma)} \FM_n^{L',\natural} \times^\Gamma T_{\gamma^{-1}} |\omega_{orb}^{d'}| \end{equation}
in $\cC_\mot (\ast_{k})$. 
\end{theorem}

Theorem \ref{mainthm} directly implies Theorem \ref{premainthm} by applying 
respectively  Lemma \ref{reduc} and Theorem \ref{equivorb} to the left hand side and to the right hand side of \eqref{mainthmeq}.

\section{Proof of Theorem \ref{mainthm}}\label{sec6}

We will write $\FA, \FA^{sm},\FP,\widehat{\FP}$, etc. \kern-0.35em for the constant pullbacks of the constructions of Section \ref{fibnot} from $k$ to $k\llb t\rrb$.

\subsection{Reduction by Fubini} We first reduce \eqref{mainthmeq} to a comparison of Hitchin fibers. For this consider the assignment  $\FA^\flat = \underline{\FA} \cap \underline{\FA^{\sm}_{k\llp t\rrp}}$ i.e. for $\overline{K}/k$ algebraically closed 
\[\FA^\flat(\overline{K}) = \FA(\overline{K}\llb t\rrb) \cap \FA^{\sm}(\overline{K}\llp t\rrp).\] 
Since the  complement of $h^{-1}(\FA^\flat)$ in $\underline{\FM^L_n}_{\circ}$ 
has smaller $K$-dimension than $\underline{\FM^L_n}_{\circ}$, it
has measure $0$, as follows from facts recalled in the first paragraph of \ref{ssf},
and thus we have by Proposition \ref{motfubini}
\[\int^{\mot}_{\underline{\FM^L_n}_{\circ}} |\omega^L| = \int^{\mot}_{h^{-1}(\FA^\flat)} |\omega^L| = \int^{\mot}_{\FA^\flat} \psi |\omega_{\FA}|, \]
where, for any point $a \in |\FA^\flat|$,
\[ \psi(a) = \int^{\mot}_{\underline{\FM^L_{n,a}}} |\omega^L_a|.\]
Here by construction, see Remark \ref{codim2}, $\omega^L_a$ is a translation invariant global form on the $\FP_a$-torsor $\FM^L_{n,a}$.
We use the properness of $h$ to identify $\underline{\FM^L_{n,a}}$
and $\underline{\FM^L_{n,a}}_{\circ}$.

%\cite{MR1085642}

Similarly the complement of $\hat{h}^{-1}(\FA^\flat)$ in $\widehat{\FM}^{d,\natural}_n$ has measure $0$ and  we can rewrite the right hand side of \eqref{mainthmeq} using Proposition \ref{eqmotfubini}:
\[ \sum_{\gamma \in \Gamma} \int^{\mot,\varrho^{-q}_\gamma}_{\ooe_{d'}^{-1}(\gamma)} \FM_n^{L',\natural}\times^\Gamma T_{\gamma^{-1}}|\omega_{orb}^{d'}| = \int^{\mot}_{\FA^\flat} \psi' |\omega_{\FA}|. \]

Here $\psi'$ is given for any $a\in |\FA^\flat|$ as
\[\psi'(a) = \sum_{\gamma \in \Gamma} \int^{\mot,\varrho^{-q}_\gamma}_{\ooe_{d',a}^{-1}(\gamma)} \underline{\FM_{n,a}^{L'} \times^\Gamma T_{\gamma^{-1}} } |\omega_{orb,a}^{d'}|,  \]
where we put
\[ \ooe_{d',a} = \ooe_{d'|\widehat{\FM}^{d}_{n,a}} :\widehat{\FM}^{d}_{n,a} \rightarrow \Gamma. \]
Combining the two sides Theorem \ref{mainthm} follows from the following

\begin{theorem}\label{fmt}
For every $a\in |\FA^\flat|$ we have
\begin{equation}\label{fibercomp}\int^{\mot}_{\underline{\FM^L_{n,a}}} |\omega^L_a| = \sum_{\gamma \in \Gamma} \int^{\mot,\varrho^{-q}_\gamma}_{\ooe_{d',a}^{-1}(\gamma)} \underline{\FM_{n,a}^{L'} \times^\Gamma T_{\gamma^{-1}} } |\omega_{orb,a}^{d'}|,\end{equation}
in $\cC_\mot (\ast_{k(a)})$
\end{theorem}

\subsection{Independence of $d$}%In a first step we rewrite Theorem \ref{mainthm} as an equality between motivic integrals. For this let $\FM^L_n$ and $\widehat{\FM}^d_n$ denote the moduli spaces defines as in Definition \ref{modsp}, but over the curve $C \times \Spec(k\llb t \rrb)$. Similarly we write $\FA, \FP, \widehat{\FP}$ for the Hitchin base and the symmetry groups over $\Spec(k\llb t \rrb)$ as defined in Section \ref{fibnot}. Notice that by functoriality all these spaces are the constant pull back of their special fibers i.e. $\FM^L_n = \M^L_n \times_{\Spec(k)} \Spec(k\llb t \rrb)$ etc.

\begin{proposition}\label{dindep} Let $L$ and $L'$ be line bundles on $C$ of degree $d$ and $d'$, both prime to $n$. Then we have for every $a\in |\FA^\flat|$
\[\int_{\underline{\FM^L_{n,a}}} |\omega^L_a| = \int_{\underline{\FM^{L'}_{n,a}}} |\omega^{L'}_a|,\]
in $\cC_\acf (\ast_{k(a)})$.
\end{proposition}

\begin{proof}
Recall that a point $a \in |\FA^\flat|$ corresponds to a morphism $\Spec(K\llb t \rrb) \to \FA$ whose generic fiber lies in $\FA^{sm}$ and where $K$ is a field extension of $k$. In particular by Proposition \ref{torfib} both $\FM^L_{n,a}$ and $\FM^{L'}_{n,a}$ are $\FP_a$-torsors over $K\llp t \rrp$ and if there is a trivialization $\tau: \FP_a \to \FM^L_{n,a}$ then $\tau^*\omega^L_a = \omega_{\FP,a}$ by construction, see Remark \ref{codim2}. Thus the proposition follows if we can show that $\FM^L_{n,a}$ has a $K\llp t \rrp$-rational point if and only if $\FM^{L'}_{n,a}$ has one. 

To see this we use the identifications $\FM^L_{n,a} \cong \Nm^{-1}(L), \FM^{L'}_{n,a} \cong \Nm^{-1}(L')$ from Proposition \ref{torfib}, where by abuse of notation we also write $L,L'$ for the line bundles on $C \times_k \Spec(K\llb t \rrb)$. We take $e$ an integer such that $de \equiv d'$ mod $n$. Then by \cite[Lemma 5.8]{gwz} we have
\[ [\Nm^{-1}(L')][\Nm^{-1}(L^{-1})]^e = [\Nm^{-1}(L'\otimes L^{-e})] \in H^1(K\llp t \rrp,\FP_a).    \]
Since the degree of the line bundle $L'\otimes L^{-e}\in \Pic(C)$ is divisible by $n$ it admits an $n$-th root (the base field $k$ is algebraically closed), from which we deduce $[\Nm^{-1}(L'\otimes L^{-e})]=0$, see again \cite[Lemma 5.8]{gwz}. This shows that $\FM^{L'}_{n,a}$ admits a $K\llp t \rrp$-rational point if $\FM^{L}_{n,a}$ does. Interchanging the roles of $d$ and $d'$ proves the converse and the proposition.
\end{proof}

\begin{corollary}Let $L$ and $L'$ be line bundles on $C$ of degree $d$ and $d'$ prime to $n$. Then we have
\[   [\M_n^{L}] = [\M_n^{L'}] \in \FM_k \otimes_{\mathbb{Z} [\mathbb{L}]} \BA.      \]
\end{corollary}

\subsection{End of proof} Because of Proposition \ref{dindep} it is enough to prove Theorem \ref{fmt} under the assumption $L=L'$. Hence we have a quotient morphism $\FM^L_{n} \rightarrow \widehat{\FM}^d_{n}$ and we may assume that $a\in \FA^\flat$ is such that $\widehat{\FM}^d_{n,a}$ is an unramified $\FPh_a$-torsor. Indeed otherwise both $\FM^L_{n,a}$ and $\widehat{\FM}^d_{n,a}$ are the empty assignment and both sides of \eqref{fibercomp} are $0$. We will write $K$ for the residue field $k(a)$.

% Notice that by Proposition \ref{torfib} the Hitchin fiber $\widehat{\FM}^{d}_{n,a}$ always has a $K\llp t \rrp$-rational point, because $\Pic^d(\FC)$ and thus $\Pic^d(\FC_a)$ does. Hence $\widehat{\FM}^{d}_{n,a}$ is a trivial $\widehat{\FP}_a$-torsor over $K\llp t \rrp$. Furthermore the fiber of the quotient map $ \FM^L_{n,a} \rightarrow \widehat{\FM}^{d}_{n,a}$ over any $x \in \widehat{\FM}^{d}_{n,a}(K\llp t \rrp)$ is a $\Gamma$-torsor $T_x$ and we have
%\begin{equation}\label{pprod}\FM^L_{n,a} \cong \FP_a \times^\Gamma T_x,\end{equation}
%by the compatibility of the actions of $\FP_a$ and $\widehat{\FP}_a$. This observation allows us to compute the left hand side of \eqref{fibercomp}.

We start by computing the left hand side of \eqref{fibercomp}.

\begin{proposition}\label{lhseasy} If $\FM^L_{n,a}$ is an unramified $\FP_a$-torsor we have
\[\int^{\mot}_{\underline{\FM^L_{n,a}}} |\omega^L_a| = \BL^{-\ord_{\FN}\omega^L_a} [\FN(\FM^L_{n,a})_K] \in \cC_\mot (\ast_{K}),\]
where $\ord_{\FN}\omega_{\omega^L_a} \in \BZ$ is the order of vanishing of $\omega^L_a$ along the special fiber of $\FN(\FM^L_{n,a})_K$, see Remark \ref{nerint}.  If $\FM^L_{n,a}$ is not unramified, then
\[\int^{\mot}_{\underline{\FM^L_{n,a}}} |\omega^L_a| = 0.\]
\end{proposition}
\begin{proof} If $\FM^L_{n,a}$ is unramified, then $\FN(\FM^L_{n,a})$ is a smooth model of $\FM^L_{n,a}$ over $\Spec(K\llb t \rrb)$. Furthermore by the N\'eron mapping property the subassignements associated with $\FM^L_{n,a}$ and $\FN(\FM^L_{n,a})$ are isomorphic, hence we can argue as in Remark \ref{nerint} to conclude. 

If $\FM^L_{n,a}$ is not unramified, $\underline{\FM}^L_{n,a}$ is the empty assignment and thus $\int^{\mot}_{\underline{\FM}^L_{n,a}} |\omega^L_a| = 0$.
\end{proof}

We are left with computing the right hand side of \eqref{fibercomp} in the two cases appearing in Proposition \ref{lhseasy}. For this let $\partial$ be the connecting homomorphism in the long exact sequence 
\begin{equation}\label{connhom} 0 \to \Gamma \to \FP_a(K\llp t \rrp) \to \FPh_a(K\llp t \rrp) \stackrel{\partial}{\rightarrow} H^1(K\llp t \rrp,\Gamma) \to \dots.\end{equation}
Recall from Remark \ref{compdir} that we have $H^1(K\llp t \rrp,\Gamma) \cong H^1(K,\Gamma) \oplus \Gamma$. With respect to this decomposition we write $\partial = \partial^{ur} \oplus \op$. Similar to Proposition \ref{edefin} one sees that $\op$ extends to a definable morphism 
\[ \op: \FPh_a \rightarrow \Gamma.  \]

Notice that $\op$ and $\ooe_{d,a}$ are closely related as the following lemma shows.

\begin{lemma}\label{unrcrit} The $\FP_a$-torsor $\FM^L_{n,a}$ is unramified if and only if  the images of $\op$ and $\ooe_{d,a}$, when evaluated over $\overline{K}$, are equal as subsets of $\Gamma$ i.e. 
\[ \im(\ooe_{d,a})(\overline{K}) =  \im(\op)(\overline{K}). \]
\end{lemma}
\begin{proof} Notice that $\im(\ooe_{d,a})(\overline{K})$ is a coset for $\im(\op)(\overline{K})$, since the actions of $\FP_a$ and $\FPh_a$ on the respective Hitchin fibers are compatible with the quotient map $\FM^L_{n} \rightarrow \widehat{\FM}^d_{n}$. The lemma now follows from the observation, that $\FM^L_{n,a}$ is unramified if and only if $\im(\ooe_{d,a})(\overline{K})$ contains the trivial torsor.
\end{proof}

%$\partial_L: \widehat{\FM}^d_{n,a}(K\llp t\rrp) \rightarrow H^1(K\llp t\rrp,\Gamma) $ be the map, sending a rational point to the class of its fiber over $\FM^L_{n,a} \rightarrow \widehat{\FM}^d_{n,a}$. Furthermore  we write 
%\[\op_L:  \widehat{\FM}^d_{n,a}(K\llp t\rrp) \rightarrow \Gamma\]
%for the composition of $\partial_L$ with the projection $H^1(K\llp t\rrp,\Gamma) \rightarrow \Gamma$ to its totally ramified part, see Remark \ref{compdir}. Also from \textit{loc. cit.} we deduce that 
%\[ e_d^{-1}(\gamma)_a = \op_L^{-1}(\gamma),      \]
%and thus we can rewrite the right hand side of \eqref{fibercomp} as

%\[\sum_{\gamma \in \Gamma} \int^{\mot,\varrho^{-q}_\gamma}_{e_{d}^{-1}(\gamma)_a} \underline{\FM_{n,a}^{L} \times^\Gamma T_{\gamma^{-1}} } |\omega_{orb,a}^{d}| = \sum_{\gamma \in \im(\op_L)} \int^{\mot,\varrho^{-q}_\gamma}_{\op_L^{-1}(\gamma)} \underline{\FM_{n,a}^{L} \times^\Gamma T_{\gamma^{-1}} } |\omega_{orb,a}^{d}|. \]

\begin{proposition}[Unramified case] \label{unrc} Assume that $\FM^L_{n,a}$ is unramified. Then we have
\[\sum_{\gamma \in \Gamma} \int^{\mot,\varrho^{-q}_\gamma}_{\ooe_{d,a}^{-1}(\gamma)} \underline{\FM_{n,a}^{L} \times^\Gamma T_{\gamma^{-1}} } |\omega_{orb,a}^{d}| = \BL^{-\ord_{\FN}\omega^L_a} [\FN(\FM^L_{n,a})_K] \in  \cC_\mot (\ast_{K}).\]
\end{proposition}

\begin{proof} We fix $\gamma \in \im(\ooe_{d,a})(\overline{K}) =  \im(\op)(\overline{K}) \subset \Gamma$ (otherwise the corresponding summand on the left hand side is $0$). By construction $\ooe_{d,a}^{-1}(\gamma) \subset \underline{\widehat{\FM}^{d}_{n,a}}$ is exactly the support of the function induced by $\FM_{n,a}^{L} \times^\Gamma T_{\gamma^{-1}}$. Thus we have 
\[  \int^{\mot,\varrho^{-q}_\gamma}_{\ooe_{d,a}^{-1}(\gamma)} \underline{\FM_{n,a}^{L} \times^\Gamma T_{\gamma^{-1}} } |\omega_{orb,a}^{d}| =  \int^{\mot,\varrho^{-q}_\gamma}_{\widehat{\FM}^{d}_{n,a}} \underline{\FM_{n,a}^{L} \times^\Gamma T_{\gamma^{-1}} } |\omega_{orb,a}^{d}|. \]
But since
\[\varrho^{-q}_\gamma(\gamma^{-1}) = \langle \gamma, \gamma \rangle^q = 1, \]
we have that the $\mu_n$-torsor $(\varrho^{-q}_{\gamma})_* T_{\gamma^{-1}}$ is trivial, and therefore the functions induced by $\FM_{n,a}^{L}$ and $\FM_{n,a}^{L} \times^\Gamma T_{\gamma^{-1}} $ have the same $\varrho^{-q}_\gamma$-isotypical component by Lemma \ref{twnotw}. 
As $\FM_{n,a}^{L}$ is unramified by assumption we get from Lemma \ref{qfint} 
\[  \int^{\mot,\varrho^{-q}_\gamma}_{\underline{\widehat{\FM}^{d}_{n,a}}} \underline{\FM_{n,a}^{L} \times^\Gamma T_{\gamma^{-1}} } |\omega_{orb,a}^{d}| =   \int^{\mot,\varrho^{-q}_\gamma}_{\underline{\widehat{\FM}^{d}_{n,a}}} \underline{\FM_{n,a}^{L} } |\omega_{orb,a}^{d}| = \chi_{*,c}^{\varrho^{-q}_\gamma}\BL^{-\ord_{\FN}\omega^L_a}  [\FN(\FM^L_{n,a})_K]. \]

%As $\FM^L_{n,a}$ is unramified we also have $\gamma \in \im(\op)(\overline{K})$ and thus the $\FP_a$-torsor $T_{\gamma^{-1}} \times^\Gamma \FP_a \in H^1(\overline{K}\llp t\rrp,\FP_a)$ is trivial. Therefore 
%\[\FM_{n,a}^{L} \times^\Gamma T_{\gamma^{-1}} \cong \FM_{n,a}^{L} \times^{\FP_a} \left(  T_{\gamma^{-1}}\times^\Gamma \FP_a \right)\]
%is unramified and we can apply Lemma \ref{qfint} to get
%\[  \int^{\mot,\varrho^{-q}_\gamma}_{\ooe_{d,a}^{-1}(\gamma)} \underline{\FM_{n,a}^{L} \times^\Gamma T_{\gamma^{-1}} } |\omega_{orb,a}^{d}| = \chi_{*,c}^{\varrho^{-q}_\gamma}\BL^{-\ord_{\FN}\omega^L_a}  [\FN(\FM^L_{n,a})_K \times^{\FN(\FP_a)_K} \FN(T_{\gamma^{-1}} \times^\Gamma \FP_a)_K]. \]
%Now Lemma \ref{imann} below implies that twisting by $\FN(T_{\gamma^{-1}} \times^\Gamma \FP_a)_K$ has no effect on the $\varrho_\gamma^{-q}$-isotypical component and we get 
%\[\chi_{*,c}^{\varrho^{-q}_\gamma}\BL^{-\ord_{\FN}\omega^L_a}  [\FN(\FM^L_{n,a})_K \times^{\FN(\FP_a)_K} \FN(T_{\gamma^{-1}} \times^\Gamma \FP_a)_K] = \chi_{*,c}^{\varrho^{-q}_\gamma}\BL^{-\ord_{\FN}\omega^L_a}  [\FN(\FM^L_{n,a})_K ]\]

On the other hand we have a decomposition 
\[[\FN(\FM^L_{n,a})_K] = \sum_{\gamma \in \Gamma} \chi_{*,c}^{\varrho_\gamma} [\FN(\FM^L_{n,a})_K ] = \sum_{\gamma \in \im(\op)(\overline{K})} \chi_{*,c}^{\varrho_\gamma} [\FN(\FM^L_{n,a})_K ],  \]
where the last equality follows from Corollary \ref{csum}. As the order of $\gamma$ divides $n$ an $q$ is prime to $n$, the map $\gamma \mapsto \gamma^{-q}$ induces a bijection on $\im(\op)$ and the proposition follows. \end{proof}

\begin{proposition}[Ramified case]\label{racase} If $\FM_{n,a}^{L}$ is not unramified we have  
\[\sum_{\gamma \in \Gamma} \int^{\mot,\varrho^{-q}_\gamma}_{\ooe_{d,a}^{-1}(\gamma)} \underline{\FM_{n,a}^{L} \times^\Gamma T_{\gamma^{-1}} } |\omega_{orb,a}^{d}|=0.\]
\end{proposition}
\begin{proof} We show that each summand on the left hand side is $0$. If $\gamma \notin \im(\ooe_{d,a})(\overline{K})$ this is clear, as then $\ooe_{d,a}^{-1}(\gamma)$ is empty, so we fix $\gamma \in \im(\ooe_{d,a})(\overline{K})$.

First we claim that $\FM_{n,a}^{L} \times^\Gamma T_{\gamma^{-1}}$ is unramified. To see this we pick a point $x\in \widehat{\FM}^d_{n,a}(\overline{K}\llp t\rrp)$ with $\ooe_{d,a}(x) = \gamma$. The fiber $T_x$ of the projection $\FM_{n,a}^{L} \rightarrow \widehat{\FM}^d_{n,a}$ is a $\Gamma$-torsor isomorphic to $T_\gamma$ and $\FM_{n,a}^{L} \cong \FP_a \times^{\Gamma} T_x$, since $\FM_{n,a}^{L}$ is a $\FP_a$ torsor. Hence we find over $\overline{K}\llp t\rrp$ 
\[\FM_{n,a}^{L} \times^\Gamma T_{\gamma^{-1}} \cong \left(\FP_a \times^{\Gamma} T_x \right) \times^\Gamma T_{\gamma^{-1}} \cong \FP_a,\]
which shows that  $\FM_{n,a}^{L} \times^\Gamma T_{\gamma^{-1}}$ is unramified. By Lemma \ref{qfint} we then get

\[\int^{\mot,\varrho^{-q}_\gamma}_{\ooe_{d,a}^{-1}(\gamma)} \underline{\FM_{n,a}^{L} \times^\Gamma T_{\gamma^{-1}} } |\omega_{orb,a}^{d}| = \chi_{*,c}^{\varrho^{-q}_\gamma}\BL^{-\ord_{\FN}\omega_{\FP,a}}[\FN(\FM_{n,a}^{L} \times^\Gamma T_{\gamma^{-1}})_K ].\]

 %and $y\in \widehat{\FM}^d_{n,a}(K\llp t\rrp)$ with $\op_L(y) = \gamma$. Under the trivialization of $\widehat{\FM}^d_{n,a}$ given by $y$, the maps $\partial$ and $\partial_L$ differ by the class of $\partial_L(y)$, in particular $\im(\op_L)$ is disjoint from $\im(\op)$. Using the isomorphism $\FM_{n,a}^{L} \cong \FP_a \times^\Gamma \partial_L(y)$ and Lemma \ref{qfint} we get
%\begin{multline*} \int^{\mot,\varrho^{-q}_\gamma}_{\op_L^{-1}(\gamma)} \underline{\FM_{n,a}^{L} \times^\Gamma T_{\gamma^{-1}} } |\omega_{orb,a}^{d}| = \\
 %\int^{\mot,\varrho^{-q}_\gamma}_{\op_L^{-1}(\gamma)} \underline{ \FP_a\times^\Gamma \partial_L^{ur}(y) } |\omega_{\FPh,a}|= \chi_{*,c}^{\varrho^{-q}_\gamma}\BL^{-\ord_{\FN}\omega_{\FPh,a}}[\FN(\FP_a)_K \times^\Gamma \partial_L^{ur}(y)_K].\end{multline*}
By Lemma \ref{unrcrit}, $\gamma$ and thus also $\gamma^{-q}$ are not in $\im(\op)(\overline{K})$, and thus Corollary \ref{csum} implies 
\[\chi_{*,c}^{\varrho^{-q}_\gamma}[\FN(\FM_{n,a}^{L} \times^\Gamma T_{\gamma^{-1}})_K]=0,\]
which is what we needed to show.
\end{proof}

\begin{proof}[Proof of Theorem \ref{fmt}] We just need to summarize the previous calculations. First if $a\in\FA^\flat$ is such that $\FM_{n,a}^{L}$ is not unramified we deduce form Propositions \ref{lhseasy} and \ref{racase}, that both side of \ref{fibercomp} are $0$. If $\FM_{n,a}^{L}$ is unramified we get similarly, that both sides of \ref{fibercomp} equal $\BL^{-\ord_{\FN}\omega^L_a} [\FN(\FM^L_{n,a})_K]$ by Propositions \ref{lhseasy} and \ref{unrc}. This proves Theorem \ref{fmt} which in turn implies Theorem \ref{mainthm}.
\end{proof}

\appendix
\section{Comparison of gerbes}

For the whole appendix we assume $k= \BC$ and use the notation of Section \ref{sec5}. In particular $L$ denotes a line bundle of degree $d$ on $C$ and $n$ is a positive integer prime to $d$. We further fix a point $c \in C$.

We start by recalling the construction in \cite[Section 3]{MR1990670} of $\mu_n$-gerbes $\alpha_L$ and $\wh{\alpha}_L$ on $\M^L_n$ and $\widehat{\BM}_n^d$ respectively. Since we assume $(n,d)=1$ there exists a universal $L$-twisted Higgs bundle 
\[ (\bs{\Es},\bs{\theta}) \rightarrow \M^L_n \times C,\]
and we denote the restriction to $ \M^L_n \times \ \{c\}$ by the same letters. The gerbe $\alpha_L$ is then defined as the gerbe of liftings of the projectivization $\BP\bs{\Es}$ of $\bs{\Es}$ to an honest $\SL_n$-bundle, i.e. for any \'etale neighborhood $p:U \to \M^L_n$ the category $\alpha_L(U)$ consists of pairs $(F,\phi)$ with $F \to U$ a vector bundle of rank $n$ and trivial determinant and $\phi: \BP F \xrightarrow{\sim} p^*\BP \bs{\Es}$ an isomorphism. Tensoring by $\mu_n$-torsors gives $\alpha_L$ the structure of a $\mu_n$-gerbe. 

Now $\wh{\alpha}_L$ on the quotient stack $\widehat{\BM}_n^d = [\M^L_n / \Gamma]$ is simply $\alpha_L$ together with a $\Gamma$-equivariant structure. This structure is induced from the $\Gamma$-equivariant structure on $\BP \bs{\Es}$ given for any $\gamma \in \Gamma=\Pic^0(C)[n]$ by tensoring with the corresponding line bundle $L_\gamma$ and thinking of $\BP \bs{\Es}$ as parametrizing lines in $\bs{\Es}$. Pulling back liftings along the map $\BP \bs{\Es} \to \BP \bs{\Es}$ induced by $\gamma$ defines the required action on $\alpha_L$.

Any $\gamma \in \Gamma$ induces a $\Gamma$-equivariant automorphism of $\alpha_{L|\M^{L,\gamma}_n}$ i.e. a $\Gamma$-equivariant $\mu_n$-torsor $T_{L,\gamma}$ on  $\M^{L,\gamma}_n$. The following proposition generalizes and relies on \cite[Proposition 8.1]{MR1990670}

\begin{aproposition} The $\mu_n$-torsor $T_{L,\gamma}$ is trivial and the $\Gamma$-equivariant structure is given by the character $\varrho^{-q}_\gamma$, where $q$ is the multiplicative inverse of $d$ modulo $n$.
\end{aproposition}
\begin{proof} Let $(E,\theta) \in \M^{L,\gamma}_n$. Since $(n,d)=1$ the Higgs bundle $(E,\theta)$ is stable and therefore the automorphism induced by $E \otimes L_\gamma \cong E$ given by a scalar. This implies that $\gamma$ acts trivially on $\BP \bs{\Es}_{|\M^{L,\gamma}_n}$ and thus preserves any lifting of $\BP \bs{\Es}$ over $\M^{L,\gamma}_n$. Hence $T_{L,\gamma}$ is trivial.

By \cite[Proposition 8.1]{MR1990670}, the $\Gamma$-equivariant structure on $T_{L,\gamma}$ is given by the character $\varrho^{-q}_\gamma$ in the case when $\gamma$ has maximal order $n$ (the extra assumption that $n$ is prime is only used later in the section of \textit{loc. cit.}). 

For a general $\gamma \in \Gamma$ let $m = \ord(\gamma)$ and $r=n/m$. Since $\Gamma \cong (\BZ/n\BZ)^{2g}$ there exists a $\tilde{\gamma}$ of order $n$ such that $\gamma = \tilde{\gamma}^r$ and we have $\M^{L,\tilde{\gamma}}_n \subset \M^{L,\gamma}_n$. Since $T_{L,\gamma}$ is trivial, its $\Gamma$-equivariant structure is determined by its restriction to $\M^{L,\tilde{\gamma}}_n$, where by 
\cite[Proposition 8.1]{MR1990670} it is given by $(\varrho^{-q}_{\tilde{\gamma}})^r = \varrho^{-q}_\gamma$. 
\end{proof}

\bibliographystyle{abbrv}
\bibliography{master}
\end{document}